\documentclass{article}
\usepackage[utf8]{inputenc}
\usepackage{amsmath}
\usepackage{amsfonts} 
\usepackage{amsthm}
\usepackage{verbatim}
\usepackage{etoolbox}
\usepackage{dirtytalk}
\usepackage{tikz-cd}
\usepackage{derivative}
\usepackage{amssymb}
\usepackage{authblk}

\theoremstyle{plain}
\newtheorem{thm}{Theorem}[section]
\newtheorem{lem}[thm]{Lemma}
\newtheorem{obs}[thm]{Observation}
\newtheorem{prop}[thm]{Proposition}
\newtheorem{cor}[thm]{Corollary}
\newtheorem{claim}{Claim}
\newtheorem{case}{Case}
\newtheorem*{claim*}{Claim}
\newtheorem{que}{Question}

\newtheorem{innerthmABC}{Theorem}
\newenvironment{thmABC}[1]
  {\renewcommand\theinnerthmABC{#1}\innerthmABC}
  {\endinnerthmABC}

\theoremstyle{definition}
\newtheorem{defi}[thm]{Definition}
\newtheorem{ex}[thm]{Example}

\theoremstyle{remark}
\newtheorem{rem}{Remark}[thm]

\AtBeginEnvironment{thm}{\setcounter{equation}{0}}
\AtBeginEnvironment{lem}{\setcounter{equation}{0}}
\AtBeginEnvironment{obs}{\setcounter{equation}{0}}
\AtBeginEnvironment{defi}{\setcounter{equation}{0}}
\AtBeginEnvironment{prop}{\setcounter{equation}{0}}
\AtBeginEnvironment{cor}{\setcounter{equation}{0}}

\AtBeginEnvironment{thm}{\setcounter{case}{0}}
\AtBeginEnvironment{lem}{\setcounter{case}{0}}
\AtBeginEnvironment{obs}{\setcounter{case}{0}}
\AtBeginEnvironment{defi}{\setcounter{case}{0}}
\AtBeginEnvironment{prop}{\setcounter{case}{0}}
\AtBeginEnvironment{cor}{\setcounter{case}{0}}

\AtBeginEnvironment{thm}{\setcounter{claim}{0}}
\AtBeginEnvironment{lem}{\setcounter{claim}{0}}
\AtBeginEnvironment{obs}{\setcounter{claim}{0}}
\AtBeginEnvironment{defi}{\setcounter{claim}{0}}
\AtBeginEnvironment{prop}{\setcounter{claim}{0}}
\AtBeginEnvironment{cor}{\setcounter{claim}{0}}

\newcommand{\R}{\mathbb{R}}
\newcommand{\RR}{\mathcal{R}}
\newcommand{\Rr}[1]{\mathcal{R}^{\underline{#1}}}
\newcommand{\I}{\mathcal{I}}
\newcommand{\Z}{\mathcal{Z}}
\newcommand{\ord}{\textnormal{ord}}
\newcommand{\indet}{\textnormal{indet}}
\newcommand{\dom}{\textnormal{dom}}
\newcommand{\codim}{\textnormal{codim}}

\title{Extensions of $k$-regulous functions from two-dimensional varieties}

\author{Juliusz Banecki}
\affil{Faculty of Mathematics and Computer Science\\
Jagiellonian University\\
Krakow, Poland\\
juliusz.banecki@student.uj.edu.pl}

\date{}

\begin{document}

\maketitle

\begin{abstract}
We prove that a $k$-regulous function defined on a two-dimensional non-singular affine variety can be extended to an ambient variety. Additionally we derive some results concerning sums of squares of $k$-regulous functions; in particular we show that every positive semi-definite regular function on a non-singular affine variety can be written as a sum of squares of locally Lipschitz regulous functions.
\end{abstract}

\section{Introduction}

\emph{In the paper real affine varieties are taken to be defined as in \cite{Bochnak}. The ring of regular functions of such a variety $X$ is denoted by $\RR(X)$. If moreover the variety is irreducible, the field of rational functions is denoted by $\R(X)$.}

We have to begin with a quick discussion on a fundamental definition in the recently founded field of \emph{regulous geometry}. The field is concerned with studying certain subclasses of continuous functions on real algebraic varieties admitting rational representations. One of the first papers attempting on classification of such functions is \cite{francuska}. In the paper, given a real affine variety $X\subset \R^n$, the authors consider two classes of functions on $X$:
\begin{enumerate}
    \item for a non-singular $X$, real valued functions on $X$ which are of class $\mathcal{C}^k$ and become regular after being restricted to some Zariski open and dense subset of $X$  (\cite{francuska}, D\'efinition 2.15),
    \item for a possibly singular $X$, real valued functions on $X$ for which there exists an extension to the entire affine space $\R^n$ satisfying the condition above (\cite{francuska}, page 137).
\end{enumerate}
In the discussed paper both of these two classes are called \emph{k-regulous}. It is believed that for a non-singular $X$ the two notions coincide, nonetheless it has not been proven yet. Thanks to \cite{real_and_p-adic} such a result is known for $k=0$, and this makes the definition of a $0$-regulous function clear. It is also a result of \cite{real_and_p-adic} that the condition $2$ implies condition $1$ for an arbitrary value of $k$. Unfortunately both of the two notions have been used as definitions of $k$-regulous functions in different papers before, causing some confusion. Because of that, before introducing our results we need to choose to which definition we will stick to.

In recent works the property $1$ is the one which is more commonly taken as the definition of a $k$-regulous function. An example of such usage is the paper \cite{open_problems}. Moreover this definition has the nice property that it is clearly invariant under biregular isomorphisms, and it does not highlight the space $\R^n$ in any special way. For these reasons we choose to use this one in the paper. Let us state this formally, for simplicity we will work only with irreducible varieties:

\begin{defi}
Let $X$ be a non-singular irreducible real affine variety. We define the ring of $k$-regulous functions, denoted by $\RR^k(X)$, to consist of all the functions satisfying property $1$. We associate such functions with their rational representations, such an association is clearly well defined and defines a ring embedding $\RR^k(X)\subset \R(X)$.
\end{defi}

The main result of our paper states that the two possible definitions coincide when the dimension is equal to two. Stating it formally:

\begin{thmABC}{A}\label{main_theorem}
Let $X$ be a two-dimensional closed non-singular irreducible subvariety of an affine non-singular irreducible variety $Y$. Every $k$-regulous function on $X$ is a restriction of some $k$-regulous function on $Y$.
\end{thmABC}

The theory developed here allows us to prove some additional results regarding $k$-regulous functions, not necessarily in dimension $2$. The most interesting ones are connected with Hilbert's $17$th Problem. It has been known for a long time that there exist positive semi-definite polynomials in more than two variables which admit \emph{bad points} i.e. they cannot be written as sums of squares of regular functions (see \cite{bad_points} for some discussion). It is desired to study whether there are some weaker regularity conditions which can be imposed on the squared functions. The most general result in this spirit known so far follows straight from the Positivestellensatz and says that a psd regular function can always be written as a sum of squares of 0-regulous functions (see \cite{regular_after_blowingups}). We are able to strengthen the result in the following way:
\begin{thmABC}{S1}\label{sos1}
Let $X$ be a non-singular irreducible affine variety. Every positive semi-definite function $f\in\RR(X)$ can be written as a sum of squares of locally Lipschitz $0$-regulous functions on $X$.
\end{thmABC}
In a similar manner we are able to give a sufficient condition for a regular function to be representable as a sum of squares of $k$-regulous functions:
\begin{thmABC}{S2}\label{sos2}
Let $X$ be a non-singular irreducible affine variety and let $f\in\RR(X)$ be a sum of $2k$-th powers of rational functions. Suppose that every zero of $f$ is also a zero of its derivatives up to order $2k$. Then it is a sum of squares in $\RR^k(X)$.
\end{thmABC}
These results cannot be improved much further; see Section \ref{Discussion} for discussion.

Lastly we significantly improve \cite[Corollary 2.14]{regular_after_blowingups} under the assumption on the dimension being equal to two:
\begin{thmABC}{B}\label{additional_theorem}
Let $m,l\geq 0$. Let $X$ be a non-singular irreducible affine variety of dimension $2$ and let $f\in\RR^k(X)$ be written as a fraction $f=\frac{p}{q}$ with $p,q\in\RR(X)$. Then
\begin{equation*}
    p^lf^m\in\RR^{kl+k+2l}(X).
\end{equation*}
\end{thmABC}
The result does not seem to admit any important applications, although it is quite surprising how high this class of smoothness is.

\section{Preliminaries}\label{Preliminaries}
\emph{From now on, unless explicitly stated otherwise, \say{variety} means a non-singular irreducible real affine variety.}

\subsection{Partial derivatives on manifolds}
Firstly we need to recall a basic notion from the theory of smooth manifolds. 

Let $M$ be a smooth manifold of dimension $k$ and $\varphi=(x_1,\dots,x_k):M\rightarrow \R^k$ a smooth mapping. Suppose that $a\in M$ is a point at which $\varphi$ is non-singular i.e. its derivative is an isomorphism between the tangent space of $M$ at $a$ and $\R^k$. Let $e_1,\dots,e_k$ be the canonical basis of $\R^k$. Then the following basis of the tangent space at $a$
\begin{equation*}
    ((d\varphi_a)^{-1}(e_1),\dots,(d\varphi_a)^{-1}(e_k))
\end{equation*}
is called the \emph{coordinate induced tangent basis of }$T_aM$. Since the tangent space is canonically identified with the space of derivations at the point $a$ this allows for a definition of partial derivatives with respect to a regular system of parameters:
\begin{defi}
Let $g:M\rightarrow \R$ be a function differentiable at $a$. We define the partial derivative of $g$ with respect to $x_i$ at the point $a$ to be equal to
\begin{equation*}
    \pdv{g}{x_i}(a):=dg_a((d\varphi_a)^{-1}(e_i)).
\end{equation*}

The gradient vector $[\pdv{g}{x_i}]_{1\leq i\leq k}$ of $g$ will be denoted by abuse of notation by $dg_{x_1,\dots,x_k}$ or just by $dg$.
\end{defi}

The definition above is of course well known, what is less known is that it translates well to the algebraic setting. 

Let now $X$ be a variety, $a$ a point of $X$, and $x_1,\dots,x_k$ a system of parameters of $\RR(X)_a$. Choose any embedding $X \subset \R^n$. We can find a Zariski open neighbourhood $B \subset \R^n$ of $a$ and a system of parameters $y_1,\dots,y_n \in \RR(\R^n)_{a}$ at $a$ such that 
\begin{itemize}
    \item $x_1,\dots,x_k\in \RR(A:=B\cap X)$
    \item $y_1,\dots,y_n\in\RR(B)$
    \item $x_i = y_i + \mathcal{I}(A)$ for $1 \leq i \leq k$
    \item $y_{k+1},\dots,y_n\in \I(A)$
    \item $\det\left(\left[\pdv{y_i}{e_j}\right]_{1\leq i,j\leq n}\right)\in\RR(B)^\ast$.
\end{itemize}

In this setting we define two functions $\varphi$ and $\psi$ by
\begin{align*}
    &\varphi: A \ni x \mapsto (x_1(x),\dots,x_k(x)) \in \R^k \\
    &\psi: B \ni x \mapsto (y_1(x),\dots,y_n(x)) \in \R^n.
\end{align*}
The set $B$ is chosen in such a way that the derivative of $\psi$ is non-singular at every point thereof. This means that the derivative $\varphi$ is non-singular at every point of $A$, so we can define the partial derivative $\pdv{g}{x_i}$ of a regular function $g\in\RR(A)$. Conveniently it turns out to be a regular function itself:
\begin{obs}
$\pdv{g}{x_i}$ is a regular function on $A$.
\begin{proof}
Let us first consider the case when $n=k$, i.e. $X=\R^n$ and $x_i=y_i$. With this assumption the Jacobian matrix $[\pdv{x_i}{e_j}]_{1\leq i,j\leq n}$ of $\varphi$ is invertible in $\RR(A)$. By the definition:
\begin{equation*}
    \pdv{g}{x_i}(x)=dg_x((d\varphi_x)^{-1}(e_i))=dg_x\cdot (d\varphi_x)^{-1}\cdot e_i\in\RR(A).
\end{equation*}
Notice that a local inverse of $\varphi$ may not be a regular function and we do not require it to be one, we solely invert the derivative of $\varphi$ and not the function itself.

The general case now follows easily: let $G$ be a regular extension of $g$ to $B$. It is clear from the definitions that $\pdv{g}{x_i}=\pdv{G}{y_i}\vert_A$, and the latter function is regular on $B$ from the preceding case.
\end{proof}
\end{obs}

Now if $f:=\frac{p}{q},\;p,q\in\RR(A)$ is a rational function which is not necessarily regular on $A$ then we naturally define its partial derivative to be equal to
\begin{equation*}
    \pdv{f}{x_i}=\frac{\pdv{p}{x_i}q-p\pdv{q}{x_i}}{q^2}\in \R(X)
\end{equation*}
It is clear that this definition does not depend on the representation $f=\frac{p}{q}$. Notice also that if $f$ is regular at the point $a$ then we can take $q$ to be non-vanishing at $a$, so $\pdv{f}{x_i}\in\RR(X)_a$.

Let us summarise the main points of this construction:
\begin{itemize}
    \item Given a variety $X$ and a regular system of parameters $(x_1,\dots,x_k)$ at $a\in X$ we define partial derivatives $\pdv{}{x_i}$, $1\leq i\leq k$ which are mappings from $\R(X)$ into itself. This definition does not depend on the embedding $X\subset \R^n$ or on any other choice we have made along the way. Actually it is not even important to highlight the point $a$, the derivatives depend solely on the system $(x_1,\dots,x_k)$.
    \item Partial derivatives of functions regular at $a$ stay regular at $a$. Since the definition does not depend on the point $a$ this is also true for any other point at which the mapping $\varphi=(x_1,\dots,x_k)$ is defined and its derivative is an isomorphism.
    \item Hence if we take $A_1$ to be a Zariski open set such that the derivative of $\varphi$ is non-singular at every point of it, then partial derivatives of functions regular on $A_1$ stay regular on $A_1$. In this case we say that \emph{$A_1$ is regular for the system $(x_1,\dots,x_k)$}. Equivalently we might state this condition as 
    \begin{equation*}
        \forall_{a'\in A_1} (x_1-x_1(a'),\dots, x_k-x_k(a'))\text{ is a system of parameters of }\RR(X)_{a'}. 
    \end{equation*} 
    \item If $X$ is a variety which is regular for some highlighted system of parameters $(x_1,\dots,x_k)$ then we might sometimes say that \emph{derivation is defined} on $X$. We will often restrict ourselves to this convenient setting in which globally regular functions stay regular under differentiation. 
\end{itemize}

\subsection{Definitions}

We need to introduce some notation used frequently later in the paper:
\begin{defi}
Let $X$ be a variety and $f\in \R(X)$ a rational function. We define $\dom(f)$ as the set $\{a\in X: f\in\RR(X)_a\}$ and $\indet(f):=X\backslash \dom(f)$.
\end{defi}

\begin{defi}
Let $X$ be a variety and let $f$ be a real valued function defined on some subset $B$ of $X$ (usually $f$ will be taken to be a rational function with $B=\dom(f)$). Let $A\subset X$ be an arbitrary set. We say that $f$ is locally bounded on $A$, or write $f\in L.B.(A)$, if for every $a\in A$ there exists a Euclidean neighbourhood $U$ of $a$ in $X$ such that $f$ is bounded on $U\cap A \cap B$.
\end{defi}

\begin{ex}
With this definition $1/x$ is locally bounded on $\R\backslash\{0\}$, although it is not locally bounded on $\R$.
\end{ex}

\begin{defi}
Let $A\subset X$. For $f,g\in L.B.(A)$ we define:
\begin{equation*}
    v_{f,A}(g):=\sup\{ \alpha\geq 0:\frac{g}{|f|^\alpha} \in L.B.(A)\}
\end{equation*}
We understand that the domain of the fraction is precisely the set of points at which $f$ and $g$ are defined and $f$ is non-zero.
\end{defi}

\begin{rem}
Observe that if $f,g$ and $A$ are semialgebraic then we must have
\begin{equation*}
   \frac{g}{|f|^{v_{f,A}(g)}}\in L.B.(A) 
\end{equation*}
Indeed, for suppose the contrary. Then by the Curve Selection Lemma we would find an arc $\varphi:(0,1)\rightarrow A\cap \dom(f)\cap \dom(g)$, with the limit $\lim_{t\rightarrow 0} \varphi(t)$ existing in $A$, such that $f\circ \varphi(t)\neq 0$ for $t\in (0,1)$ and 
\begin{equation*}
    \lim_{t\rightarrow 0} \frac{|g|}{|f|^{v_{f,A}(g)}}\circ \varphi(t)=\infty
\end{equation*}
Now, applying Łojasiewicz's Inequality (\cite[Proposition 2.6.4]{Bochnak}) to $\frac{|f|^{v_{f,A}(g)}}{|g|}\circ \varphi$ and $\frac{1}{|f|}\circ \varphi$ we get that for $N>0$ large enough we have
\begin{equation*}
    \lim_{t\rightarrow 0} \frac{|f|^{v_{f,A}(g)N-1}}{|g|^N}\circ \varphi(t)=0.
\end{equation*}
This shows that for $\epsilon>0$ small enough 
\begin{equation*}
    \lim_{t\rightarrow 0} \frac{|g|}{|f|^{v_{f,A}(g)-\epsilon}}\circ \varphi(t)=\infty
\end{equation*}
which contradicts the definition of $v_{f,A}(g)$.
\end{rem}

Let now $X$ be a variety, we define two following subrings of $\R(X)$:
\begin{defi}
$\Rr{0}(X)$ is defined to consist of all the rational functions which are locally bounded on $X$. 
\end{defi}

\begin{defi}
$\Rr{1}(X)\subset \RR^0(X)$ is defined to consist of all the $0$-regulous functions which are locally Lipschitz on $X$.
\end{defi}

\begin{rem}
The above definition is stated only for the sake of Theorem \ref{sos1}.
\end{rem}

\begin{obs}\label{sub1_to_sub0}
Let $X$ be a variety regular for some system of parameters $x_1,\dots,x_k$. Let $f\in\R(X)$. Then $f\in\Rr{1}(X)$ if and only if all the partial derivatives of $f$ belong to $\Rr{0}(X)$.
\begin{proof}
Implication in the one direction is easy as every partial derivative of a differentiable Lipchitz function is bounded by its Lipschitz constant. 

Assume now that the partial derivatives of $f$ all belong to $\Rr{0}(X)$. We firstly would like to show that the set $\indet(f)$ is of codimension at least $2$. Suppose the contrary and take $Z$ to be an irreducible component of the set of codimension $1$. Let $\RR(X)_Z$ be the local ring of the subvariety $Z$, i.e. the localisation of $\R(X)$ at $\I(Z)$. It is known to be a discrete valuation ring, let $q$ generate its maximal ideal. By our assumption $\ord_q(f)<0$ so we can represent $f$ as 
\begin{equation*}
    f=\frac{p}{q^\alpha}
\end{equation*}
for some $\alpha>0$, $p\in\RR(X)_Z^\ast$. Now notice that the derivative of $q$ does not vanish at a generic point of $Z$, so $q$ does not divide $\pdv{q}{x_i}$ for some $i$. Computing the derivative of $f$ we get
\begin{equation*}
    \pdv{f}{x_i}=\frac{\pdv{p}{x_i}q-\alpha p\pdv{q}{x_i}}{q^{\alpha+1}}
\end{equation*}
As the numerator is not divisible by $q$ we have $\pdv{f}{x_i}\not\in\RR(X)_Z$ so the derivative diverges at a generic point of $Z$, a contradiction.

We now basically have to reapply the reasoning in the proof of \cite[Lemma 2.16]{regular_after_blowingups}. Choose a point $a\in\indet(f)$. We find a small neighbourhood $U$ of $a$, diffeomorphically transformed into a convex set by $(x_1,\dots,x_k)$, and such that the derivative of $f$ is bounded on $U\cap\dom(f)$. Since the indeterminacy locus of $f$ is of codimension at least two, every two points of $U\cap\dom(f)$ can be connected by a curve contained in the set of length slightly longer than the distance between. By the mean value theorem this implies that $f$ is Lipschitz on $U\cap\dom(f)$, so it extends in a unique way to a Lipschitz function on $U$. As $a$ was arbitrary the conclusion follows. 
\end{proof}
\end{obs}

\begin{cor}\label{regulous_if_derivatives}
Let $X$ be a variety on which derivation is defined. Let $f\in\R(X)$. Then $f\in\RR^k(X)$ if and only if $\frac{\partial^k f}{\partial x_{i_1}\dots\partial x_{i_k}}\in\RR^0(X)$ for every $1\leq i_1,\dots,i_k\leq n$.
\begin{proof}
Implication in the one direction is trivial, if $f$ is $k$-regulous then the partial derivatives of $f$ of order $k$ are rational continuous functions. In the other direction we proceed by induction on $k$. By Observation \ref{sub1_to_sub0} and the inductive assumption we have $f\in\Rr{1}(X)\subset \RR^0(X)$. Now we just need to repeat the proof of \cite[Theorem 2.12]{regular_after_blowingups} in the new slightly more general setting.
\end{proof}
\end{cor}

We will now define the central concept of the paper, namely \emph{$\RR^k$-flat} functions.

\begin{defi}
Let $k\geq 1$ be a natural number and let $a$ be a point of a variety $X$. We say that a fraction $\frac{p}{q}$ with $p,q\in\RR(X)_a$ is an $\RR^k$-flat representation at $a$ if there exists a Zariski open neighbourhood $A$ of $a$ on which derivation is defined, such that $p,q\in\RR(A)$ and:
\begin{center}
    The function $\frac{p||dq||^k}{q^{k+1}}$ approaches zero at all points of the set $\Z(q)\cap A$.
\end{center}
Obviously this condition is independent on the choice of the system of parameters on $A$. Notice that by Łojasiewicz's inequality (\cite[Proposition 2.6.4]{Bochnak}) applied to $\frac{p||dq||^k}{q^{k+1}}$ and $\frac{1}{q}$ the condition may be equivalently stated as 
\begin{equation*}
    v_{q,A}(p||dq||^k)>k+1.
\end{equation*}
In a completely similar fashion we say that $\frac{p}{q}$ is an $\Rr{1}$-flat representation at $a$ if $\frac{p||dq||}{q^2}\in L.B.(A)$ (which is equivalent to $v_{q,A}(p||dq||)\geq 2$).
\end{defi}

\begin{defi}
A rational function $f$ is said to be locally $\RR^k$-flat at $a$ (resp. locally $\Rr{1}$-flat at $a$) if it admits an $\RR^k$-flat (resp. $\Rr{1}$-flat) representation as a fraction at $a$.
$f$ is said to be $\RR^k$-flat ($\Rr{1}$-flat) if it is locally $\RR^k$-flat ($\Rr{1}$-flat) at every point of $X$.
\end{defi}

\begin{defi}
An $\RR^k$-flat function $f$ is said to have a global representation if there exist two regular functions $p,q\in \RR(X)$ such that $\frac{p}{q}$ is an $\RR^k$-flat representation of $f$ at every point of $X$.
\end{defi}
It is natural to ask if every $\RR^k$-flat function admits a global representation. We are unaware whether the answer to this question is affirmative although in a moment we will prove something only slightly weaker. Before that we need to state a lemma a proof of which will be given at the beginning of the next section:
\begin{lem}\label{to_be_proven_later}
Let $\frac{p}{q}$ be an $\RR^k$-flat representation on some Zariski open set $A$ on which derivation is defined. Then $v_{q,A}(p)>1$.
\end{lem}

\begin{cor}
If $f$ is locally $\RR^k$-flat at $a\in X$ then it admits an $\RR^k$-flat representation $f=\frac{p}{q}$ at $a$ with $p$ and $q$ both belonging to $\RR(X)$.
\begin{proof}
Take $f=\frac{p}{q}$ to be any $\RR^k$-flat representation of $f$ on some neighbourhood $A$ of $a$ on which derivation is defined. Take $r\in \RR(X)$ to be such that $pr,qr\in\RR(X)$ and $\Z(r)\subset X\backslash A$. We claim that $\frac{pr}{qr}$ is an $\RR^k$-flat representation of $f$ on $A$. A direct computation shows that:
\begin{equation*}
    \left|\frac{pr\left(\pdv{qr}{x_i}\right)^k}{(qr)^{k+1}}\right|=\left|\frac{p}{q}\left(\frac{\pdv{q}{x_i}}{q}+\frac{\pdv{r}{x_i}}{r}\right)^k\right|
    \leq 2^{k-1}\left|\frac{p\left(\pdv{q}{x_i}\right)^k}{q^{k+1}}\right|+2^{k-1}\left|\frac{p\left(\pdv{r}{x_i}\right)^k}{qr^k}\right|
\end{equation*}
where the first summand approaches zero on $\Z(q)\cap A$ by the fact that $\frac{p}{q}$ is an $\RR^k$-flat representation and the second one does so by Lemma \ref{to_be_proven_later}.
\end{proof}
\end{cor}

\begin{prop}\label{local_to_global_rk}
Let $f$ be a function locally $\RR^k$-flat at $a\in X$. Then there exists $g\in \RR(X)$ with $g(a)\neq 0$ such that $gf$ is $\RR^k$-flat and has a global representation.
\begin{proof}
Let $\frac{p}{q}$ be an $\RR^k$-flat representation of $f$ on a neighbourhood $A$ of $a$ with $p,q\in\RR(X)$. Let $g\in\RR(X)$ satisfy $\Z(g)=X\backslash A$. Choose any $a'\in X\backslash A$ and consider some local system of parameters $x_1.\dots,x_n$ at $a'$ together with a regular neighbourhood $B$ of $a'$. The function $\frac{p||dq_{x_1,\dots,x_n}||^k}{q^{k+1}}$ is semialgebraic and it approaches $0$ at every point of $A\cap B\cap \Z(q)$. This way by Łojasiewicz's inequality (\cite[Proposition 2.6.4]{Bochnak}) we can find a large exponent $N$ such that $g^N\frac{p||dq_{x_1,\dots,x_n}||^k}{q^{k+1}}$ approaches zero at every point of $B\backslash A$. In particular $\frac{pg^N}{q}$ is an $\RR^k$-flat representation at every point of $B\cup A$. By Noetherianity of the Zariski topology we can find an exponent $N$ working for all the points at once.
\end{proof}
\end{prop}

\begin{cor}
Every $\RR^k$-flat function is a finite sum of $\RR^k$-flat functions admitting global representations.
\begin{proof}
Let $f$ be $\RR^k$-flat. Consider the ideal 
\begin{equation*}
    I=\{g\in\RR(X): \text{ $fg$ is a sum of $\RR^k$-flat functions with global representations}\}
\end{equation*}
By the last proposition $\Z(I)=\emptyset$, so the Weak Nullstellensatz finishes the proof.
\end{proof}
\end{cor}

\subsection{Plan of the paper}
After these preparations we can begin the main course of the paper. Almost all the material here is devoted to the proof of Theorem \ref{main_theorem}, the remaining results come more as byproducts of the reasoning. Our general approach somewhat resembles the approach of Koll\'ar and Nowak presented in \cite{real_and_p-adic}, up to the fact that dealing with higher classes of smoothness requires much more subtleties. The main idea is that we want to prove that Theorem \ref{main_theorem} holds for functions which are $\RR^k$-flat, and from that extrapolate the result to all $k$-regulous functions. To achieve this we follow the following plan:

In Section \ref{Divisorial valuations} we prove that every $\RR^k$-flat function is $k$-regulous. The proof requires us to study a certain interesting construction associating a set of valuations to a rational locally bounded function. We then derive an inequality regarding such valuations, from which the conclusion of this part follows.

In Section \ref{Extending RR^k-flat functions} we show that every $\RR^k$-flat function can be extended to a sum of $\RR^k$-flat functions on an ambient variety. The proof is heavily inspired by the proof of \cite[Proposition 10]{real_and_p-adic} and is completely explicit. Unfortunately the current general scenario forces the reasoning to be way more technical than the one in \cite{real_and_p-adic}.

In Section \ref{The differential multiplicity and sums of squares} we study some behaviour of regulous functions when composed with a sequence of blowups. This lays foundations for the reasoning of Section \ref{RR^k-flat decomposition in dimension 2}. At the end of the section we diverge from the main course of the proof for a moment to prove Theorems \ref{sos1} and \ref{sos2}.

In Section \ref{RR^k-flat decomposition in dimension 2} we restrict ourselves to two dimensions to show that every $k$-regulous function on a two-dimensional variety can be decomposed as a finite sum of $\RR^k$-flat functions. The proof relies on a close inspection of how the problem behaves after applying a resolution of singularities of the $k$-regulous function. The decomposition is a recursive algorithm, producing a summand for each step of the resolution. This section is definitely the most technical part of the paper.

In two short Sections \ref{An additional result} and \ref{Discussion} we prove Theorem \ref{additional_theorem} using the machinery developed so far, and then we discuss possible further improvements and generalisations of the results.

Let us formally derive the proof of Theorem \ref{main_theorem} using the announced results:
\begin{proof}[Proof of Theorem \ref{main_theorem}]
Let $X$ be a closed two-dimensional subvariety of a variety $Y$ and let $f\in\RR^k(X)$. By Theorem \ref{rk_flat_decomposition} $f$ can be written as a sum of $\RR^k$-flat functions on $X$. By Corollary \ref{rk_flat_extension} each of them can be extended to a sum of $\RR^k$-flat functions on $Y$. Lastly by Corollary \ref{rk_falt_is_kregulous} all of the extending functions are $k$-regulous on $Y$, so their sum is the extension of $f$ we are looking for.
\end{proof}
%---------------------------------------------------------------------------------------------------------------------------------------------------------------------
\section{Divisorial valuations}\label{Divisorial valuations}

\subsection{The valuation inequality}
We begin by introducing the central object of study of this section:
\begin{defi}
Let $X$, $X_\sigma$ be varieties and let $\sigma:X_\sigma\rightarrow X$ be a proper, birational morphism. Let $D$ be a non-singular irreducible codimension one subvariety of $X_\sigma$. The localisation $\RR(X_\sigma)_D$ induces a discrete valuation on $\R(X_\sigma)\cong\R(X)$, which is denoted by $\ord_D$. Every valuation on $\R(X)$ constructed in such a way is called a \emph{divisorial valuation}.
\end{defi}

We now describe a certain scenario in which such valuations arise naturally.

Let $f\in\RR(X)$ be a simple normal crossing on $X$, we would like to describe the function $v_{f,X}(g)$ as $g$ ranges through $\RR(X)$. Notice that a necessary condition for $\alpha$ to be such that $\frac{g}{|f|^\alpha}\in L.B.(X)$ is that the order of $g$ at every irreducible component $D$ of $\Z(f)$ is at least $\alpha$ times the order of $f$ at the component, otherwise the fraction would diverge at a generic point of $D$. On the other hand, if the condition
\begin{equation*}
    \alpha \ord_D(f)\leq \ord_D(g)
\end{equation*}
is satisfied for every such $D$ then analysing what happens locally at every point of $X$ we see that indeed $\frac{g}{|f|^\alpha}\in L.B.(X)$. This allows us to characterise the function $v_{f,X}$ using divisorial valuations:
\begin{equation}\label{rees_decomposition}
    v_{f,X}(g)=\min_i \frac{\ord_{D_i}(g)}{\ord_{D_i}(f)}
\end{equation}
where the minimum is taken over all valuations induced by the irreducible components of $\Z(f)$.

Consider now a more general question of describing $v_{f,X}(g)$ as $g$ ranges through $\RR(X)$ and $f$ is a non-zero rational locally bounded function on $X$. In this case by the resolution of singularities we can find a proper birational morphism $\sigma:X_\sigma\rightarrow X$
such that $f\circ \sigma$ is regular and a simple normal crossing on $X_\sigma$. Since $\sigma$ is proper a short topological consideration shows that it induces an isomorphism between the rings $\Rr{0}(X)$ and $\Rr{0}(X_\sigma)$. Hence the equation (\ref{rees_decomposition}) applies also here, this time the set of valuations ranges through the orders at the irreducible components of $\Z(f\circ\sigma)$. Let us emphasise this as an observation:

\begin{obs}
For a given $f\in\Rr{0}(X)\backslash\{0\}$ there exists a finite set of divisorial valuations $\ord_{D_1},\dots,\ord_{D_k}$ such that the equation (\ref{rees_decomposition}) holds for every $g\in\RR(X)$.
\end{obs}

As a trivial corollary we get $v_{f,X}(g)\in \mathbb{Q}$.

The valuations constructed this way are the Rees valuations associated to the ideal $(f\circ\sigma)$ in the ring $\RR(X_\sigma)$ (see \cite[Example 10.1.2]{swanson-huneke}). We cannot guarantee that the constructed set is minimal in the sense that after deleting any of them the equality fails for some $g\in\RR(X)$, but we can always alter it. 

Assume that we have deleted the unnecessary valuations. Then, by \cite[Theorem 10.1.6]{swanson-huneke} the set $\ord_{D_1},\dots,\ord_{D_k}$ is uniquely determined by $f$; it does not depend in any way on the resolution of singularities we have used to find it.

\begin{defi}
The valuations $\ord_{D_i}$ in the minimal decomposition (\ref{rees_decomposition}) are called the divisorial valuations associated to $f$. 
\end{defi}

Let us now see a few examples:

\begin{ex}
The function $x^2+y^2\in\RR(\R^2)$ admits precisely one associated valuation $\ord_{(x,y)}$. In general functions of the form $x^{2a}+y^{2b}$ admit precisely one associated monomial valuation with weights $x\equiv b,y\equiv a$ (up to normalisation). An example of a function with two associated valuations is $(x^2+y^2)(x^2+y^4)$, its valuations come from the two factors. A less trivial example is the following:
\begin{equation*}
    (x^2+y^2)(x^2+z^2)\in\RR(\R^3)
\end{equation*}
There are three monomial valuations associated to it, the orders at $(x,y),(x,z)$ and at the point $(x,y,z)$.
\end{ex}

It is natural to ask about the converse, i.e. whether every divisorial valuation arises as a valuation associated to some locally bounded $f$. It turns out that it does so:

\begin{prop}\label{existance_of_q_prop}
For every divisorial valuation $\ord_D$ there exists $f\in\Rr{0}(X)$ such that $\ord_D$ is the only valuation associated to $f$, moreover it can be chosen so that $\ord_D(f)=2$ and $f$ is positive semi-definite.
\begin{proof}
Take a variety $X_\sigma$ with a proper birational morphism $\sigma:X_\sigma\rightarrow X$ such that $\ord_D$ is induced as the order at a non-singular divisor $D$ of $X_\sigma$. Choose $h$ to be the sum of squares of generators of the ideal $\I(D)$ in $\RR(X_\sigma)$. Then we can take $f:=h\circ\sigma^{-1}\in\Rr{0}(X)$. The equation
\begin{equation*}
    v_{f,X}(g)=v_{h,X_\sigma}(g\circ\sigma)=\frac{\ord_D(g)}{2}
\end{equation*}
becomes clear after analysing what happens locally at every point of $D$.
\end{proof}
\end{prop}

\begin{prop}\label{some_ineq_deriv}
Suppose that $X$ is regular for some system of parameters $x_1,\dots,x_n$. Let $f\in\RR(X)\backslash\{0\}$ and let $\ord_D$ be a divisorial valuation with $\ord_D(f)>0$. Then
$\ord_D(f)> \ord_D(\pdv{f}{x_i})$ holds for some $i$.
\begin{proof}
Take $g\in\Rr{0}(X)$ such that $\ord_D$ is the only valuation associated to $g$. Define
\begin{equation*}
    c:=v_{g,X}(||df||)
\end{equation*}
It is clear that $c=\min_i v_{g,X}(\pdv{f}{x_i})$, so our goal is to show that $v_{g,X}(f)>c$. If $c=0$ the conclusion is trivial, so we assume not.

Choose a smooth bounded semialgebraic curve $\xi:(0,1)\rightarrow \dom(g)$ with bounded derivative, with the limit of $g\circ\xi$ at zero being equal to zero. Then the limits of both $f\circ\xi$ and $||df||\circ\xi$ at zero must also be zero. We may restrict ourselves to a small subinterval $(0,\epsilon)$ such that $||df||\circ \xi$ is weakly increasing there. By the mean value theorem we get:
\begin{align*}
    |f(\xi(t))|\leq Ct\sup_{0< t_1\leq t}||df||(\xi(t_1))&\leq Ct||df||(\xi(t))\\
    \frac{|f|}{|g|^c}(\xi(t))\leq Ct\frac{||df||}{|g|^c}(\xi(t))&\xrightarrow[]{t\to 0} 0
\end{align*}
Since the choice of the curve was arbitrary the Curve Selection Lemma implies that $v_{g,X}(f)>c$.
\end{proof}
\end{prop}

\begin{proof}[Proof of Lemma \ref{to_be_proven_later}]
Let $\ord_D$ be a divisorial valuation associated to $q\vert_A$. By the definition of $\RR^k$-flatness we have 
\begin{multline*}
    (k+1)\ord_D(q)<\min_i \ord_D\left(p\left(\pdv{q}{x_i}\right)^k\right)=\ord_D(p)+\min_i k\ord_D\left(\pdv{q}{x_i}\right)<\\
    <\ord_D(p)+k\ord_D(q)\\
\end{multline*}
so $\ord_D(p)>\ord_D(q)$. As this is true for every such valuation, the conclusion follows.
\end{proof}

The main tool of the section is the following:
\begin{thm}[valuation inequality]\label{valuation_inequality}
Suppose that $X$ is regular for some system of parameters $x_1,\dots,x_n$. Let $f\in\RR(X)$, $h\in \R(X)^*$ and let $\ord_D$ be a divisorial valuation associated to $f$. Then
\begin{equation*}
    \min_i \ord_D\left(\pdv{h}{x_i}\right)-\ord_D(h)\geq \min_i \ord_D\left(\pdv{f}{x_i}\right)-\ord_D(f).
\end{equation*}
\end{thm}

Before presenting the proof of the inequality we exhibit the important corollaries it brings about.
\begin{cor}
Every $\Rr{1}$-flat function $f$ belongs to $\Rr{1}(X)$.
\begin{proof}
Choose a Zariski open subset $A\subset X$ on which derivation is defined and where $f$ admits an $\Rr{1}$-flat representation $f=\frac{p}{q}$. On $A$, $v_{q,A}$ admits a decomposition into divisorial valuations associated to $q$. For every such valuation $\ord_D$, for every $i$ Theorem \ref{valuation_inequality} gives us
\begin{equation*}
   \ord_D\left(\pdv{p}{x_i}\right)-\ord_D(p)\geq \min_j\ord_D\left(\pdv{q}{x_j}\right)-\ord_D(q) .
\end{equation*}
On the other hand by our assumption on $f$
\begin{equation*}
    \min_j \ord_D(p)+\ord_D\left(\pdv{q}{x_j}\right)\geq 2\ord_D(q)
\end{equation*}
so altogether
\begin{equation*}
    \ord_D\left(\pdv{p}{x_i}\right)\geq \ord_D(q).
\end{equation*}
This implies that $\pdv{p}{x_i}/q\in\Rr{0}(A)$ and hence 
\begin{equation*}
    \pdv{f}{x_i}=\frac{\pdv{p}{x_i}}{q}-\frac{p\pdv{q}{x_i}}{q^2}\in \Rr{0}(A).
\end{equation*}
As $X$ can be covered by such sets $A$ the conclusion follows.
\end{proof}
\end{cor}

\begin{cor}\label{rk_falt_is_kregulous}
Every $\RR^k$-flat function $f$ is $k$-regulous. 
\begin{proof}
We work by induction on $k$. The case $k=1$ is treated analogously to the preceding proof. Assume $k>1$ and as before choose $A$ to be regular for some system of parameters, $\frac{p}{q}$ to be an $\RR^k$-flat representation of $f$ on $A$ and let $\ord_D$ be a valuation associated to $q\vert_A$.
\begin{claim*}
Every partial derivative of $f$ is $k-1$-regulous on $A$.
\begin{proof}
We have
\begin{equation*}
    \pdv{f}{x_i}=\frac{\pdv{p}{x_i}}{q}-\frac{p\pdv{q}{x_i}}{q^2}
\end{equation*}
The second summand is $\RR^{k-1}$-flat by the fact that $f$ is $\RR^k$-flat, so it is $k-1$-regulous by induction. To prove the same for the first one, first apply Theorem \ref{valuation_inequality} to obtain for every $i$:
\begin{equation*}
    \ord_D\left(\pdv{p}{x_i}\right)-\ord_D(p)\geq \min_j\ord_D\left(\pdv{q}{x_j}\right)-\ord_D(q).
\end{equation*}
By the assumption on the $\RR^{k}$-flatness of $\frac{p}{q}$ we have
\begin{equation*}
    \ord_D(p)+k\min_j \ord_D\left(\pdv{q}{x_j}\right)>(k+1)\ord_D(q).
\end{equation*}
The two equations together imply
\begin{equation*}
    \ord_D\left(\pdv{p}{x_i}\right)+(k-1)\min_j \ord_D\left(\pdv{q}{x_j}\right)>k\ord_D(q).
\end{equation*}
As this is true for every valuation $\ord_D$ associated with $q\vert_A$, we get that
\begin{equation*}
    v_{q,A}\left(\pdv{p}{x_i}\left(\pdv{q}{x_j}\right)^{k-1}\right)>k
\end{equation*}
holds for every $j$ and hence that $\frac{\pdv{p}{x_i}}{q}$ is $\RR^{k-1}$-flat. This proves the claim and the conclusion follows.
\end{proof}
\end{claim*}
\end{proof}
\end{cor}

%---------------------------------------------------------------------------------------------------------------------------------------------------
\subsection{Proof of the valuation inequality}
We now head towards the proof of Theorem \ref{valuation_inequality}. We introduce the following notation:

\begin{defi}
If $\psi,\xi:(0,1)\rightarrow \R$ then we write $\psi\preceq \xi$ if and only if there exists a constant $C>0$ such that $|\psi|\leq C|\xi|$ holds on some small subinterval $(0,\epsilon)$. Notice that whenever $\psi$ and $\xi$ are semialgebraic, they are comparable in this sense.
\end{defi}

\begin{lem}
\label{r_polynomial_lemma}
Let $k\in\mathbb{N}$ and $\psi_0,\dots,\psi_k:(0,1)\rightarrow \R$ be semialgebraic functions. Then for all but finitely many $r\in\R$ we have
\begin{equation*}
    \max(|\psi_0|,\dots,|\psi_k|)\preceq \psi_0+r\psi_1+r^2\psi_2+\dots+r^k\psi_k.
\end{equation*}

\begin{proof}
The set $\{\psi_0,\dots,\psi_k\}$ must contain a maximal element in the order $\preceq$, say $\psi_{k_0}$. Consider
\begin{equation*}
    \lim_{t\rightarrow0}\frac{\psi_0(t)}{\psi_{k_0}(t)}+r\frac{\psi_1(t)}{\psi_{k_0}(t)}+\dots+r^k\frac{\psi_k(t)}{\psi_{k_0}(t)}.
\end{equation*}
Since $\forall_i\lim_{t\rightarrow 0} \psi_i(t)/\psi_{k_0}(t)\in\R$ the expression above is a polynomial as a function of r, moreover its coefficient next to $r^{k_0}$ is equal to 1 so it is non-zero. Therefore for all but finitely many $r$ the limit is equal to a real number different than zero. In other words for every such $r$
\begin{equation*}
    \psi_{k_0}\preceq \psi_0+r\psi_1+\dots+r^k\psi_k.
\end{equation*}
On the other hand once again because of the maximality of $\psi_{k_0}$ we have
\begin{equation*}
    \max(|\psi_0|,\dots,|\psi_k|)\preceq \psi_{k_0}.
\end{equation*}
Combining the two equations finishes the proof.
\end{proof}
\end{lem}

\begin{lem}\label{valuation_ineq_auxiliary}
Suppose that $X$ is regular for a system of parameters $x_1,\dots,x_n$. Let $f,g,h\in\RR(X)\backslash \{0\}$ and let $\ord_D$ be a divisorial valuation with $\ord_D(f)>0$. Define
\begin{equation*}
    \alpha:=\sup\left\{\frac{\ord_D(f)-\ord_D\left(\frac{\partial^k f}{\partial x_{i_1}\dots\partial x_{i_k}}\right)}{k}:1\leq k,1\leq i_1,\dots,i_k \leq n\right\}.
\end{equation*}
Then, for every $1\leq i\leq n$ the following inequality holds:
\begin{equation*}
    \ord_D\left(\pdv{h}{x_i}g\right)\geq \ord_D(f)v_{f,X}(hg)-\alpha.
\end{equation*}

\begin{proof}
Take $q\in\Rr{0}(X)$ such that $v_{q,X}=\frac{1}{2}\ord_D$ and $q$ is psd (such $q$ exists thanks to Proposition \ref{existance_of_q_prop}). We need to show that 
\begin{equation*}
    \frac{||dh||g}{q^{\frac{1}{2}(\ord_D(f)v_{f,X}(hg)-\alpha)}}\in L.B.(X).
\end{equation*}
This question is local so we may consider it on a small neighbourhood $U$ of a given point $a\in X$. Firstly we choose it to be so small that $(x_1,\dots,x_n)$ is a parameterisation of $U$. We can apply a linear change of coordinates so that the three functions $f,g,h$ are regular with respect to all the variables at the point $a$, such a transformation will not affect the magnitude of $||dh||$ nor the value of $\alpha$. Using the Weierstrass Preparation Theorem we can locally present the three functions in the forms
\begin{align*}
    f=f_1f_2 \\
    g=g_1g_2 \\
    h=h_1h_2  
\end{align*}
where $f_1,g_1,h_1$ are Nash functions which are polynomial with respect to the variable $x_1$, meanwhile $f_2,g_2,h_2$ are Nash functions which do not vanish at the point $a$.

By the Curve Selection Lemma it is enough to show that for a given smooth semialgebraic function $\psi:[0,1)\rightarrow U$ with $\psi((0,1))\subset \dom(q)$ for all $i$ we have
\begin{equation*}
    \pdv{h}{x_i}(\psi)g(\psi)q^{\frac{1}{2}\alpha}(\psi)\preceq q^{\frac{1}{2}\ord_D(f)v_{f,X}(hg)}(\psi).
\end{equation*}
By symmetry we may prove it only for $i=1$. Let $r\in\R$. We define 
\begin{equation*}
    \xi(t)=(\psi_1(t)+rq^{\frac{1}{2}\alpha}(\psi(t)),\psi_2(t),\dots,\psi_n(t)).
\end{equation*}
Since $q^{\frac{1}{2}\alpha}$ is locally bounded on $U$ ($\alpha> 0$ by Proposition \ref{some_ineq_deriv}) we have that if $r$ is sufficiently close to $0$ and we restrict ourselves to some small interval $(0,\epsilon)$ then $\xi$ is a well defined function into $U$, which can be continuously extended to $t=0$.

To simplify notation let $^{(i)}$ describe $i$-th order partial derivative purely in the direction of the first variable. Consider:
\begin{equation*}
    h_1(\xi)=\sum_{0\leq i\leq \deg(h_1)}\frac{r^iq^{\frac{1}{2}i\alpha}h_1^{(i)}}{i!}(\psi).
\end{equation*}
Applying Lemma \ref{r_polynomial_lemma} we see that we can choose such a value of $r$ that 
\begin{align}
    h_1(\psi),h_1^{(1)}(\psi)q^{\frac{1}{2}\alpha}(\psi)&\preceq h_1(\xi) \nonumber \\
    h^{(1)}(\psi)q^{\frac{1}{2}\alpha}(\psi)=[h_1^{(1)}(\psi)h_2(\psi)+h_1(\psi)h_2^{(1)}(\psi)]q^{\frac{1}{2}\alpha}(\psi)&\preceq h(\xi). \label{preceq_ineq1}
\end{align}
Similarly we might simultaneously require the property
\begin{equation}\label{preceq_ineq2}
      g(\psi)\preceq g(\xi) 
\end{equation}

Now let us focus our attention on the function $f$.
By the definition of $\alpha$ we get that for all $i$
\begin{equation*}
    \frac{(f_1f_2)^{(i)}}{q^{\frac{1}{2}(\ord_D(f)-i\alpha)}}\in L.B.(U).
\end{equation*}
This allows us to inductively show that 
\begin{equation*}
     \frac{f_1^{(i)}}{q^{\frac{1}{2}(\ord_D(f)-i\alpha)}}\in L.B.(U).
\end{equation*}
This means that
\begin{align}
    \frac{f_1(\xi)}{q^{\frac{1}{2}\ord_D(f)}(\psi)}&=\sum_{0\leq i\leq \deg(f_1)}\frac{r^if_1^{(i)}}{i!q^{\frac{1}{2}(\ord_D(f)-i\alpha)}}(\psi)\preceq 1 \nonumber\\
    f(\xi)&\preceq q^{\frac{1}{2}\ord_D(f)}(\psi). \label{preceq_ineq3}
\end{align}

Lastly, by the definition
\begin{equation}\label{preceq_ineq4}
    h(\xi)g(\xi)\preceq f^{v_{f,X}(hg)}(\xi).
\end{equation}

The equations (\ref{preceq_ineq1}), (\ref{preceq_ineq2}), (\ref{preceq_ineq3}), (\ref{preceq_ineq4}) together imply
\begin{equation*}
    h^{(1)}(\psi)g(\psi)q^{\frac{1}{2}\alpha}(\psi)\preceq h(\xi)g(\xi)\preceq f^{v_{f,X}(hg)}(\xi)\preceq q^{\frac{1}{2}\ord_D(f)v_{f,X}(hg)}(\psi)
\end{equation*}
which was to be shown.
\end{proof}
\end{lem}

We are now ready to prove Theorem \ref{valuation_inequality}:

\begin{proof}[Proof of Theorem \ref{valuation_inequality}]
It is enough to prove the theorem when $h\in\RR(X)$ since if $h=\frac{h_1}{h_2}$ then
\begin{equation*}
    \pdv{h}{x_i}/h=\pdv{h_1}{x_i}/h_1-\pdv{h_2}{x_i}/h_2.
\end{equation*}
Let $\alpha$ be defined as in Lemma \ref{valuation_ineq_auxiliary}. 
\begin{claim}\label{proof_val_claim1}
\begin{equation*}
     \ord_D\left(\pdv{h}{x_i}\right)-\ord_D(h)\geq -\alpha\text{ for all }i.
\end{equation*}
\begin{proof}
Let $\ord_{D_1},\dots,\ord_{D_m}$ be the valuations associated to $f$. After reindexing we assume $\ord_D=\ord_{D_1}$. By the minimality of the representation we can find some $g\in\RR(X)$ such that
\begin{equation*}
    \frac{\ord_D(g)}{\ord_D(f)}<\min_{i\neq 1}\frac{\ord_{D_i}(g)}{\ord_{D_i}(f)}.
\end{equation*}
Hence we can find $N$ big enough so that
\begin{equation*}
    \frac{\ord_D(hg^N)}{\ord_D(f)}<\min_{i\neq 1}\frac{\ord_{D_i}(hg^N)}{\ord_{D_i}(f)}.
\end{equation*}
This means that
\begin{equation*}
    v_{f,X}(hg^N)=\min_i\frac{\ord_{D_i}(hg^N)}{\ord_{D_i}(f)}=\frac{\ord_D(hg^N)}{\ord_D(f)}.
\end{equation*}
Now the conclusion of Lemma \ref{valuation_ineq_auxiliary} turns into
\begin{align*}
    \ord_D\left(\pdv{h}{x_i}g^N\right)\geq \ord_D(f)v_{f,X}(hg^N)-\alpha&=\ord_D(hg^N)-\alpha \\
    \ord_D\left(\pdv{h}{x_i}\right)-\ord_D(h)&\geq -\alpha.
\end{align*}
\end{proof}
\end{claim}

\begin{claim}
\begin{equation*}
    \alpha=\max_i \left(\ord_D(f)-\ord_D\left(\pdv{f}{x_i}\right)\right)
\end{equation*}
\begin{proof}
The definition of $\alpha$ (or the preceding claim) gives us inequality in one direction, combining it with Proposition \ref{some_ineq_deriv} we get
\begin{equation*}
    \alpha\geq\max_i \left(\ord_D(f)-\ord_D\left(\pdv{f}{x_i}\right)\right)> 0
\end{equation*}
so in particular $\alpha>0$. Suppose that the first inequality is also strict. Choose any $k\geq 1$ and $1\leq i_1,\dots,i_k\leq n$. Repeatedly applying Claim \ref{proof_val_claim1} we get:
\begin{multline*}
    \ord_D(f)-\ord_D\left(\frac{\partial^k f}{\partial x_{i_1}\dots\partial x_{i_k}}\right)\leq \ord_D(f)-\ord_D\left(\frac{\partial^{k-1} f}{\partial x_{i_1}\dots\partial x_{i_{k-1}}}\right)+\alpha\leq\\
    \leq\dots\leq \ord_D(f)-\ord_D\left(\pdv{f}{x_{i_1}}\right)+(k-1)\alpha < k\alpha .
\end{multline*}
Hence
\begin{equation*}
    \frac{ \ord_D(f)-\ord_D\left(\frac{\partial^k f}{\partial x_{i_1}\dots\partial x_{i_k}}\right)}{k}<\alpha
\end{equation*}
which means that $\alpha$ is not attained in the supremum in its definition. It contradicts the fact that $\alpha>0$ even though
\begin{equation*}
    \frac{\ord_D(f)-\ord_D\left(\frac{\partial^k f}{\partial x_{i_1}\dots\partial x_{i_k}}\right)}{k}\leq\frac{\ord_D(f)}{k}\rightarrow 0
\end{equation*}
as $k\rightarrow \infty$.
\end{proof}
\end{claim}
The proof of the theorem follows by combining both claims.
\end{proof}

%---------------------------------------------------------------------------------------------------------------------------------------------------------------------
\section{Extending $\RR^k$-flat functions}\label{Extending RR^k-flat functions}
This section is devoted to the proof of the fact that every $\RR^k$-flat function on a subvariety extends to a sum of $\RR^k$-flat functions on an ambient variety. We will need the following:

\begin{lem}\label{flat_extension_lemma}
Let $A$ be a codimension one subvariety of a variety $B$, and $x_1,\dots,x_n$ be a regular system of parameters of $\R(B)_a$ at a point $a\in A$ satisfying $(x_1)=\I(A)$. Suppose the entire variety $B$ is regular for the system $x_1,\dots,x_n$. Then for every $g\in \RR(A)$ and $M\in\mathbb{N}$ there exists $G_M\in \RR(B)$ satisfying
\begin{enumerate}
    \item $g=G_M\vert_A$ 
    \item $\frac{\partial^m G_M}{\partial x_1^m}\in \I(A)$ for any $1 \leq m \leq M$.
\end{enumerate}
\begin{proof}
Firstly take $G_0$ to be any regular extension of $g$ to $B$, we will improve it inductively. Having defined $G_{M-1}$ it suffices to consider
\begin{equation*}
    G_M:=G_{M-1}-\frac{x_1^M}{M!}\frac{\partial^M G_{M-1}}{\partial x_1^M}.
\end{equation*}
\end{proof}
\end{lem}

We firstly prove the main result in a local version:
\begin{thm}\label{local_extension_theorem}
Let $X$ be a closed subvariety of a variety $Y$. Let $f\in\R(X)$ be locally $\RR^k$-flat at $a\in X$. Then there exists a rational function $F$ on $Y$, locally $\RR^k$-flat at $a$ such that $f$=$F\vert_X$ (the equality making sense at all points of $X$ at which $f$ is defined).
\begin{proof}
Since the question is local we may for simplicity assume that $\codim(X)=1$ and in the general case construct an extension one dimension at a time.
Take $\frac{p}{q}$ to be an $\RR^k$-flat representation of $f$ at $a$. Choose a system of parameters $x_1,\dots,x_n$ of $\RR(Y)_a$ such that locally $(x_1)=\I(X)$. Choose a small Zariski neighbourhood $B\subset Y$ of $a$ such that it is regular for $x_1,\dots,x_n$, $(x_1)=\I(A:=B\cap X)$ holds in $\RR(B)$, and such that $\frac{p}{q}$ is an $\RR^k$-flat representation of $f$ on $A$. 

Now by Lemma \ref{to_be_proven_later} we can find a small $\alpha\in\mathbb{Q}$ with 
\begin{equation*}
    0<k\alpha<v_{q,A}(p)-1\text{ and } \alpha<1.
\end{equation*}

Choose also some $N\in 2\mathbb{N}$ satisfying $N\alpha>2k+2$, $N\alpha\in 2\mathbb{N}$.
Now apply Lemma \ref{flat_extension_lemma} to find $P,Q\in\RR(B)$ with
\begin{equation*}
    p=P\vert_A, \;q=Q\vert_A, 
\end{equation*}
\begin{equation} \label{Q1R1_are_flat}
    \frac{\partial^m P}{\partial x_1^m}, \frac{\partial^m Q}{\partial x_1^m}\in \I(A) \text{ for all }1\leq m \leq \frac{2k+6}{\alpha}.
\end{equation}
Notice that differentiating an element of $\I(A)$ with respect to any of the last $n-1$ variables will not make it leave $\I(A)$, so we also have
\begin{equation}
    \frac{\partial^m \pdv{Q}{x_i}}{\partial x_1^m}\in \I(A) \text{ for all }1\leq m \leq \frac{2k+5}{\alpha},\;1\leq i\leq n.
\end{equation}
Define:
\begin{align*}
    P_1&:=PQ^{N\alpha-1} \\
    Q_1&:=Q^{N\alpha}+x_1^N
\end{align*}
It is obvious that $f=\frac{P_1}{Q_1}\vert_X$, we will show that $\frac{P_1}{Q_1}$ is an $\RR^k$-flat representation on $B$ which will finish the proof. Fix some $b\in \Z(Q_1)\cap B$; our task is to show that $\frac{P_1||dQ_1||^k}{Q_1^{k+1}}$ approaches zero at $b$. In other words we need to prove that:
\begin{equation*}
    \frac{PQ^{N\alpha-1}}{Q^{N\alpha}+x_1^N}\left(\frac{N\alpha Q^{N\alpha-1}\pdv{Q}{x_i}+Nx_1^{N-1}\delta_{1,i}}{Q^{N\alpha}+x_1^N}\right)^k\rightarrow 0\text{ as }x\rightarrow b\text{ for all } i.
\end{equation*}
We have the following elementary inequality:
\begin{multline*}
    \left|N\alpha\frac{\pdv{Q}{x_i}}{Q}\right|^k+\left|N\frac{1}{|Q|^\alpha}\right|^k\geq \\
    \geq \left|\frac{N\alpha Q^{N\alpha-1}\pdv{Q}{x_i}}{Q^{N\alpha}+x_1^N}\right|^k+ \left|N\frac{1}{(Q^{N\alpha}+x_1^N)^\frac{1}{N}}\frac{|x_1|^{N-1}}{(Q^{N\alpha}+x_1^N)^\frac{N-1}{N}}\right|^k \geq\\
    \geq 2^{1-k}\left|\frac{N\alpha Q^{N\alpha-1}\pdv{Q}{x_i}+Nx_1^{N-1}\delta_{1,i}}{Q^{N\alpha}+x_1^N}\right|^k
\end{multline*}
Hence it is enough to show that
\begin{equation}\label{extension_goal}
    \frac{PQ^{N\alpha-1}}{Q^{N\alpha}+x_1^N}\left(\frac{\pdv{Q}{x_i}}{Q}\right)^k\rightarrow 0,\; \frac{PQ^{N\alpha-1}}{Q^{N\alpha}+x_1^N}\frac{1}{|Q|^{k\alpha}}\rightarrow 0\text{ as }x\rightarrow b\text{ for all }i. \tag{*}
\end{equation}
Take $U$ to be a small Euclidean neighbourhood of $b$ in $B$, such that the function $\varphi:=(x_1,\dots,x_n)$ restricted to $U$ is a diffeomorphism onto some n-dimensional hypercube. We cover $U$ by two sets $U_1$ and $U_2$, which can be interpreted as points lying "far away" from the set $X\backslash \Z(q)$ and "close" to the set $X\backslash\Z(q)$ respectively:
\begin{align*}
    U_1&:=\{x\in U: x_1^2\geq |Q(x)|^\alpha\} \\
    U_2&:=\{x\in U: x_1^2\leq |Q(x)|^\alpha\} 
\end{align*}
On $U_1$ our goal is not difficult to reach:
\begin{align*}
    \left|\frac{PQ^{N\alpha-1}}{Q^{N\alpha}+x_1^N}\left(\frac{\pdv{Q}{x_i}}{Q}\right)^k\right|&\leq \left|P\left(\pdv{Q}{x_i}\right)^k\frac{Q^{\frac{1}{2}N\alpha}}{x_1^N}Q^{\frac{1}{2}N\alpha-1-k}\right|\rightarrow 0\\
    \left|\frac{PQ^{N\alpha-1}}{Q^{N\alpha}+x_1^N}\frac{1}{|Q|^{k\alpha}}\right|&\leq\left|P\frac{Q^{\frac{1}{2}N\alpha}}{x_1^N}|Q|^{\frac{1}{2}N\alpha-1-k\alpha}\right|\rightarrow 0.
\end{align*}

On $U_2$ things get a bit messier. The point here is that since $U_2$ contains only points lying sufficiently close to the variety $X$, the functions $P$, $Q$ and the partial derivatives of $Q$ all behave roughly as they do when we restrict them to $X$, and on $X$ we have everything under control.

Since $\varphi(U)$ is a hypercube it makes sense to define the following retraction:
\begin{equation*}
    \pi:U\rightarrow U\cap A, \; (x_1,x_2,\dots,x_n)\mapsto (0,x_2,\dots,x_n)
\end{equation*}
From Taylor's theorem and (\ref{Q1R1_are_flat}) we get for $x\in U_2$:
\begin{equation*}
    ||Q(x)-Q\circ \pi(x)||\leq C||(x_1,0,\dots,0)||^{(2k+4)/\alpha}=C|x_1|^{(2k+4)/\alpha}\leq C|Q(x)|^{k+2}
\end{equation*}
and so
\begin{equation}\label{valuation_of_Rpi1}
    v_{Q,U_2}(Q-Q\circ \pi)> k+1.
\end{equation}
Similarly for $P$ and all first order partial derivatives of $Q$ we get
\begin{align}
    v_{Q,U_2}(P-P\circ \pi)&>k+1 \label{valuation_of_Q}\\
    \forall_i \;v_{Q,U_2}\left(\pdv{Q}{x_i}-\pdv{Q}{x_i}\circ \pi\right)&>k+1 \label{valuation_of_R'} .
\end{align}
The equation (\ref{valuation_of_Rpi1}) implies that 
\begin{equation}\label{valuation_of_Rpi}
    v_{Q,U_2}(Q\circ \pi)\geq 1.
\end{equation}
Now notice that 
\begin{equation*}
    v_{Q\circ \pi,U_2}(P\circ \pi)=v_{q,\pi(U_2)}(p)>1+k\alpha.
\end{equation*}
By (\ref{valuation_of_Rpi}) we get $v_{Q,U_2}(P\circ \pi)>1+k\alpha$, so by (\ref{valuation_of_Q})
\begin{equation*}
    v_{Q,U_2}(P)=v_{Q,U_2}((P-P\circ \pi)+P\circ \pi)>1+k\alpha.
\end{equation*}
This proves the second part of (\ref{extension_goal}) knowing that $b\in U_2$:
\begin{equation*}
     \left|\frac{PQ^{N\alpha-1}}{Q^{N\alpha}+x_1^N}\frac{1}{|Q|^{k\alpha}}\right|\leq\left|\frac{P}{|Q|^{1+k\alpha}}\right|\rightarrow 0.
\end{equation*}
%---------------------------------------------------------------------------------------

Now let $1<i\leq n$ and consider
\begin{equation*}
    v_{Q\circ \pi,U_2}\left((P\circ \pi)\left(\pdv{Q}{x_i}\circ \pi\right)^k\right)=v_{q,\pi(U_2)}\left(p\left(\pdv{q}{x_i}\right)^k\right)>k+1.
\end{equation*}
Notice that for $i=1$ we have $\pdv{Q}{x_1}\circ \pi \equiv 0$ so the above inequality trivially stays true. Hence by (\ref{valuation_of_Rpi}) 
\begin{equation}\label{valuation_of_QpiR'pi}
    v_{Q,U_2}\left((P\circ \pi)\left(\pdv{Q}{x_i}\circ \pi\right)^k\right)>k+1.
\end{equation}
From Newton's Binomial Formula it follows that
\begin{multline*}
   P\left(\pdv{Q}{x_i}\right)^k=\\
    =(P-P\circ\pi)\left(\pdv{Q}{x_i}\right)^k+
    (P\circ\pi)\sum_{j=0}^k\binom{k}{j} \left(\pdv{Q}{x_i}\circ\pi\right)^j\left(\pdv{Q}{x_i}-\pdv{Q}{x_i}\circ\pi\right)^{k-j}
\end{multline*}
Using (\ref{valuation_of_Q}), (\ref{valuation_of_R'}) and (\ref{valuation_of_QpiR'pi}) this gives us
\begin{equation*}
    v_{Q,U_2}\left(P\left(\pdv{Q}{x_i}\right)^k\right)>k+1.
\end{equation*}
This implies the first part of (\ref{extension_goal}) finishing the proof:
\begin{equation*}
    \left|\frac{PQ^{N\alpha-1}}{Q^{N\alpha}+x_1^N}\left(\frac{\pdv{Q}{x_i}}{Q}\right)^k\right|\leq \left|\frac{P\left(\pdv{Q}{x_i}\right)^k}{Q^{k+1}}\right|\rightarrow 0.
\end{equation*}

\end{proof}
\end{thm}
\begin{cor}[$\RR^k$-flat extension]\label{rk_flat_extension}
Let $X$ be a subvariety of a variety $Y$ and let $f$ be an $\RR^k$-flat function on $X$. Then $f$ extends to a finite sum of $\RR^k$-flat functions on $Y$.
\begin{proof}
By Theorem \ref{local_extension_theorem} and Proposition \ref{local_to_global_rk} for every $a\in X$ we can find a regular function $g$ such that $g(a)\neq0$ and $fg$ extends to a globally $\RR^k$-flat function on $Y$. This means that
\begin{equation*}
    \Z(\{g\in\RR(X):gf\text{ extends to a sum of $\RR^k$-flat functions on Y}\})=\emptyset
\end{equation*}
so the Weak Nullstellensatz finishes the proof.
\end{proof}
\end{cor}

%------------------------------------------------------------------------------------------------------------------------------------------------------------------------------
\section{The differential multiplicity}\label{The differential multiplicity and sums of squares}
\emph{In this section we use some basic terminology regarding resolution of singularities, we refer the reader to \cite{resolution_snc} for the definitions.}

From now on we will almost exclusively work in a setting where $X$ and $X_\sigma$ are varieties of the same dimension $n$ and $\sigma:X_\sigma\rightarrow X$ is the composition of an admissible sequence of blowups. In such a case, for short we will say that $\sigma$ is an \emph{admissible multiblowup}. By the \emph{exceptional divisors} of $\sigma$ we mean the irreducible components of its exceptional locus; every such divisor arises as the exceptional divisor of one of the blowups in the sequence. The multiblowup $\sigma$ induces an isomorphism between $\R(X)$ and $\R(X_\sigma)$; we identify the two fields using it.

\subsection{Restricted simple normal crossings}

\begin{defi}
Let $X$ be a variety and $\sigma:X_\sigma\rightarrow X$ an admissible multiblowup. Let $f$ be a rational function regular at a point $a_\sigma\in X_\sigma$. Then, $f$ is said to be \emph{rsnc (restricted simple normal crossing)} with respect to $\sigma$ at $a_\sigma$ if there exists a system of parameters $u_1,\dots,u_n$ at $a_\sigma$ and a number $0\leq k\leq n$, such that for every $1\leq i \leq k$, $u_i$ locally generates an exceptional divisor of $\sigma$ and $f$ can be written as
\begin{equation*}
    f=u_1^{\alpha_1}\dots u_k^{\alpha_k}r
\end{equation*}
for some $\alpha_1,\dots,\alpha_k\geq 0$ and $r\in\RR(X_\sigma)_{a_\sigma}^*$. We say that $f$ is rsnc on $X_\sigma$ (with respect to $\sigma$) if it is rsnc at every point thereof.
\end{defi}

Firstly we need to prove a few lemmas regarding restricted simple normal crossings.
\begin{prop}\label{rsnc_stays_after_blowup}
Let $\sigma_1:X_{\sigma_1}\rightarrow X_\sigma$ be another multiblowup as in the diagram below, so that $\sigma\circ\sigma_1$ is admissible.
\begin{center}
    \begin{tikzcd}
    X & X_{\sigma} \arrow[l, "\sigma"] & X_{\sigma_1} \arrow[l, "\sigma_1"]
    \end{tikzcd}
\end{center}
Let $f\in \R(X)$ be rsnc at $a_\sigma\in X_\sigma$ (with respect to $\sigma$). Then $f$ is rsnc at every point of $\sigma_1^{-1}(a_\sigma)$ (with respect to $\sigma\circ\sigma_1$).
\begin{proof}
By induction it is enough to consider the case when $\sigma_1$ is a blowup. Ignoring the trivial cases, we may assume that the center of $\sigma_1$ contains $a_\sigma$, and that $a_\sigma$ lies in the exceptional locus of $\sigma$. By our assumption $f$ can be written as
\begin{equation*}
    f=u_1^{\alpha_1}\dots u_k^{\alpha_k}r
\end{equation*}
so to it is enough to consider the cases when $f=u_i$ for some $i$ or $f=r$. The latter case is trivial since then $f$ is a unit at every point of $\sigma^{-1}(a_\sigma)$, so without loss of generality we take $f=u_1$. Let the exceptional divisor $\Z(u_1)$ of $\sigma$ be denoted by $D$. We again consider two possibilities:

If the center of $\sigma_1$ is not contained in $D$ then since the blowup is admissible $u_1$ locally generates $D$ at every point of $\sigma_1^{-1}(a_\sigma)$.

If the center of $\sigma_1$ is contained in $D$ then we have to consider each of the blowup charts separately. There are two types of them, the first one with coordinate system of the form
\begin{equation*}
    \begin{cases}
        u_1=v_jv_1 \\
        \dots\\
        u_{j-1}=v_jv_{j-1} \\
        u_j=v_j \\
        u_{j+1}=v_{j+1}\\
        \dots \\
        u_n=v_n
    \end{cases}
\end{equation*}
for some $j>1$. In this case $v_j$ locally generates the exceptional divisor induced by $\sigma_1$ at every point of $\sigma_1^{-1}(a_\sigma)$, meanwhile $v_1$ locally generates $D$. The other case is
\begin{equation*}
    \begin{cases}
        u_1=v_1\\
        u_2=v_1v_2\\
        \dots\\
        u_j=v_1v_j\\
        u_{j+1}=v_{j+1}\\
        \dots\\
        u_n=v_n
    \end{cases}
\end{equation*}
Here $v_1$ also locally generates the newly created divisor at every point of $\sigma_1^{-1}(a_\sigma)$ so we are done.
\end{proof}
\end{prop}

\begin{lem}\label{rel_prime_set}
Let $a\in X$ and let $p,q$ be two relatively prime elements of $\RR(X)_a$. There exists a Zariski open subset $a\in A\subset X$ such that
\begin{enumerate}
    \item $p,q\in\RR(A)$
    \item $p$ and $q$ are relatively prime in every localization $\RR(A)_{a'}$ for $a'\in A$.
\end{enumerate}
\begin{proof}
Writing $q$ as a product of primes in $\RR(X)_a$ it is enough to consider each of the primes separately, so we assume that $q$ is a prime itself. Choose a Zariski open neighbourhood $B$ of $a$ with $p,q\in\RR(B)$ and consider the ideal
\begin{equation*}
    I:=\{f\in \RR(B):q\text{ divides }f\text{ in }\RR(X)_a\}.
\end{equation*}
Since it is finitely generated we can find an even smaller neighbourhood $A$ of $a$ so that
\begin{equation*}
    f\in I \implies \frac{f}{q}\in\RR(A).
\end{equation*}

Suppose that $p$ and $q$ are not relatively prime at a point $a'\in A$. Since $\RR(X)_{a'}$ is a UFD this means that there exists $p'\in\RR(X)_{a'}$ satisfying 
\begin{align*}
    \frac{p'}{q}&\not\in\RR(X)_{a'}\\
    \frac{pp'}{q}&\in\RR(X)_{a'}
\end{align*}
There exists $g\in\RR(B)$ satisfying $g(a')\neq 0$ and $gp',\frac{pgp'}{q}\in\RR(B)$. Since $q$ does not divide $p$ in $\RR(X)_a$, $q$ must divide $gp'$ in that ring, so by the definition of $A$ we have 
\begin{equation*}
    \frac{gp'}{q}\in\RR(A)\subset\RR(X)_{a'}
\end{equation*}
Hence $\frac{p'}{q}\in\RR(X)_{a'}$; a contradiction. 
\end{proof}
\end{lem}

\begin{lem}\label{gcd_lemma}
Let $a\in X$ and let $p,q$ be two relatively prime elements of $\RR(X)_a$. Then for any admissible multiblowup $\sigma:X_\sigma\rightarrow X$ and a point $a_\sigma$ with $\sigma(a_\sigma)=a$, the greatest common divisor of $p$ and $q$ in $\RR(X_\sigma)_{a_\sigma}$ is rsnc at $a_\sigma$.
\begin{proof}
By induction it is enough to prove the lemma when $\sigma$ is a blowup. Choose local coordinates $x_1,\dots,x_n$ at $a$ so that $\Z(x_1,\dots,x_j)$ is its center (for some $1\leq j\leq n$). Without loss of generality $a_\sigma$ belongs to the following chart:
\begin{equation}\label{affine_chart_def}
    \begin{cases}
        x_1=u_1\\
        x_2=u_2u_1\\
        \dots\\
        x_j=u_ju_1\\
        x_{j+1}=u_{j+1}\\
        \dots\\
        x_n=u_n
    \end{cases}
\end{equation}
Here $u_1$ generates the newly created exceptional divisor, so we are to prove that the greatest common divisor of $p$ and $q$ in $\RR(X_\sigma)_{a_\sigma}$ is a monomial in $u_1$. 

The affine chart, let it be called $A$, is defined by the equations (\ref{affine_chart_def}) together with equations defining $X$. This means that if we denote its coordinate ring by $\mathcal{P}(A)$ then we have the following inclusions:
\begin{equation}\label{inclusions}
    \RR(X)_a[\frac{1}{x_1}]\supset\RR(X)_a[u_1,\dots,u_n]\supset \mathcal{P}(A)\subset\mathcal{P}(A)_{a_\sigma}= \RR(X_\sigma)_{a_\sigma}
\end{equation}
Let $r$ be the greatest common divisor of $p$ and $q$ in $\RR(X_\sigma)_{a_\sigma}$. This means that there exist $p', q'\in\RR(X_\sigma)_{a_\sigma} $ such that
\begin{align*}
    rp'=p \\
    rq'=q
\end{align*}

By the last equality in (\ref{inclusions}) we can represent $r$ as $r=\frac{r_1}{r_2}$ where $r_1,r_2\in \mathcal{P}(A)$ and $r_2\in\mathcal{P}(A)_{a_\sigma}^\ast= \RR(X_\sigma)_{a_\sigma}^\ast$. One can find similar representations $p'=\frac{p_1}{p_2}$ and $q'=\frac{q_1}{q_2}$. This way after multiplying by a large enough power of $x_1$ we get:
\begin{align*}
    (r_1x_1^N)(p_1x_1^N)=(r_2x_1^N)(p_2x_1^N)p\\
    (r_1x_1^N)(q_1x_1^N)=(r_2x_1^N)(q_2x_1^N)q
\end{align*}
with $r_ix_1^N,p_ix_1^N,q_ix_1^N\in \RR(X)_a$. Since $p$ and $q$ are relatively prime we get that $r_1x_1^N$ divides $r_2p_2q_2x_1^{3N}$ in $\RR(X)_a$. This implies that it divides $u_1^{3N}$ in $\RR(X_\sigma)_{a_\sigma}$, so it is rsnc at $a_\sigma$. Hence $r$ is also one.
\end{proof}
\end{lem}

\begin{prop}\label{monomial_denom_representation}
Let $\sigma:X_\sigma\rightarrow X$ be an admissible multiblowup and let $f\in \RR(X_\sigma)$. Let $a\in X$. If $\frac{p}{q}$ is an irreducible representation of $f=\frac{p}{q}$ as a fraction in $\RR(X)_a$ then $q$ is rsnc at every point of $\sigma^{-1}(A)$ for some Zariski neighbourhood $A$ of $a$.
\begin{proof}
Using Lemma \ref{rel_prime_set} we can take $A$ to be such that $p$ and $q$ are relatively prime in $\RR(X)_{a'}$ for every $a'\in A$. Lemma \ref{gcd_lemma} finishes the proof, as $q$ divides $p$ in $\RR(X_\sigma)_{a_\sigma}$ for every $a_\sigma\in\sigma^{-1}(a')$.
\end{proof}
\end{prop}

%------------------------------------------------------------------

\subsection{The differential multiplicity}
Let us now introduce partial derivatives to the setting considered. Suppose we have a system of parameters $x_1,\dots,x_k$ at a point of $X$ and a system of parameters $u_1,\dots,u_k$ at a point of $X_\sigma$. An element $f\in \R(X)$ by our definitions admits partial derivatives with respect to the variables $x_1,\dots,x_n$. Using the association $\R(X)\cong \R(X_\sigma)$ we might as well speak of $\pdv{f}{u_1},\dots,\pdv{f}{u_n}\in \R(X)$. 

\begin{prop}[The chain rule]
The different derivatives of $f$ are related by the following formulas:
\begin{equation*}
    \begin{bmatrix}
        \pdv{f}{x_1}   & \cdots & \pdv{f}{x_n}
    \end{bmatrix}
    \begin{bmatrix}
        \pdv{x_1}{u_1} & \cdots & \pdv{x_1}{u_n} \\
        \vdots         & \ddots & \vdots         \\
        \pdv{x_n}{u_1} & \cdots & \pdv{x_n}{u_n} \\
    \end{bmatrix}
    =
    \begin{bmatrix}
        \pdv{f}{u_1}   & \cdots & \pdv{f}{u_n}
    \end{bmatrix}
\end{equation*}

\begin{equation*}
    \begin{bmatrix}
        \pdv{f}{x_1}   & \cdots & \pdv{f}{x_n}
    \end{bmatrix}
    =
    \begin{bmatrix}
        \pdv{f}{u_1}   & \cdots & \pdv{f}{u_n}
    \end{bmatrix}
    \begin{bmatrix}
        \pdv{u_1}{x_1} & \cdots & \pdv{u_1}{x_n} \\
        \vdots         & \ddots & \vdots         \\
        \pdv{u_n}{x_1} & \cdots & \pdv{u_n}{x_n} \\
    \end{bmatrix}
\end{equation*}
\begin{proof}
We can find a non-empty Zariski open subset of $X$ on which $\sigma$ is an isomorphism, the systems $x_1,\dots,x_n$ and $u_1,\dots,u_n$ are regular and $f$ is regular. After such a restriction the equations are just the usual chain rule (written in two directions).
\end{proof}
\end{prop}

\begin{obs}\label{diff_localmon}
Let $a\in X,\;a_\sigma\in\sigma^{-1}(a)$ and $f\in\RR(X_\sigma)_{a_\sigma}$. Let $x_1,\dots,x_n$ be a system of parameters at $a$. There exists $g\in\RR(X_\sigma)_{a_\sigma}$ being rsnc at $a_\sigma$ such that $g\pdv{f}{x_1}\in\RR(X_\sigma)_{a_\sigma}$. More precisely if all the exceptional divisors passing through $a_\sigma$ are $D_1,\dots,D_k$ and $u_1,\dots,u_k$ generate their ideals respectively, then we can take
\begin{equation*}
    g=u_1^{\max(0,-\ord_{D_1}(\pdv{f}{x_1}))}\dots u_k^{\max(0,-\ord_{D_k}(\pdv{f}{x_1}))}.
\end{equation*}
\begin{proof}
Take $f=\frac{q}{r}$ to be an irreducible representation of $f$ in $\RR(X)_a$. Then $r^2\pdv{f}{x_1}\in\RR(X)_a\subset\RR(X_\sigma)_{a_\sigma}$ so we can take $g=r^2$ which is rsnc by Proposition \ref{monomial_denom_representation}. Since the valuations $\ord_{D_i}$ are the $u_i$-adic valuations in the ring $\RR(X_\sigma)_{a_\sigma}$, we can then reduce $g$ by powers of these elements until no further reduction is possible. The orders of $g$ in $u_1,\dots,u_k$ must then be equal to $\max(0,-\ord_{D_1}(\pdv{f}{x_1})),\dots,\max(0,-\ord_{D_k}(\pdv{f}{x_1}))$ respectively.
\end{proof}
\end{obs}

\begin{prop}\label{val_diff_ineq}
Let $u_1,\dots,u_n$ be a system of parameters at $a_\sigma\in X_\sigma$. Let $D$ be the divisor induced by the local blowup with center $C=\Z(u_1,\dots,u_k)$ for some $1\leq k\leq n$. Let $f\in\R(X)$ and $1\leq i\leq n$. Then
\begin{align*}
    \ord_D\left(\pdv{f}{u_i}\right)&\geq \ord_D(f)-1 \text{ for }i\leq k\\
    \ord_D\left(\pdv{f}{u_i}\right)&\geq \ord_D(f) \text{ for }i>k
\end{align*}
\begin{proof}

Write $f$ as $f=\frac{p}{q}$ with $p,q\in \RR(X_\sigma)_C$. Since
\begin{equation*}
     \ord_D\left(\frac{\pdv{f}{u_i}}{f}\right)=\ord_D\left(\frac{\pdv{p}{u_i}}{p}-\frac{\pdv{q}{u_i}}{q}\right)\geq \min\left(\ord_D\left(\frac{\pdv{p}{u_i}}{p}\right),\ord_D\left(\frac{\pdv{q}{u_i}}{q}\right)\right)
\end{equation*}
it is enough to prove the proposition in the cases $f=p$ and $f=q$, so we assume $f\in \RR(X_\sigma)_C$.

Let $\mathfrak{m}$ denote the maximal ideal of $\RR(X_\sigma)_C$. It is in general true for a regular local ring that the order at the divisor $D$ induced by the blowup at $\mathfrak{m}$ is equal to the order $\ord_\mathfrak{m}$ (See \cite[Example 10.3.1]{swanson-huneke}). Hence we just need to check that
\begin{align*}
    \ord_\mathfrak{m}\left(\pdv{f}{u_i}\right)&\geq \ord_\mathfrak{m}(f)-1 \text{ for }i\leq k,\\
    \ord_\mathfrak{m}\left(\pdv{f}{u_i}\right)&\geq \ord_\mathfrak{m}(f) \text{ for }i>k.
\end{align*}
It now suffices to consider these inequalities when $f$ is a generator of $\mathfrak{m}^{\ord_\mathfrak{m}(f)}$, i.e. a monomial in the variables $u_1,\dots,u_k$. In this case they become completely trivial.
\end{proof}
\end{prop}

\begin{defi}
Let $\sigma:X_\sigma\rightarrow X$ be an admissible multiblowup and let $D$ be a divisor on $X_\sigma$. Let $a\in\sigma(D)$ and let $x_1,\dots,x_n$ be a system of parameters at $a$. We define the following quantity, which we call \emph{the differential multiplicity} of $D$:
\begin{equation*}
    k_D:=\sup \left(\ord_D(f)-\min_i\ord_D\left(\pdv{f}{x_i}\right)\right)
\end{equation*}
where the supremum is taken over all $f\in\R(X)$.
\end{defi}

\begin{rem}
This definition is clearly related to the valuation inequality from Section \ref{Divisorial valuations}, however we are not able to make much use of this fact.
\end{rem}

At first it may seem that this notion is not well defined, let us show that it is not the case. 

Firstly if $y_1,\dots,y_n$ is another system of parameters at $a$ then by the chain rule
\begin{equation*}
    \ord_D\left(\pdv{f}{x_i}\right)=\ord_D\left(\sum_j\pdv{f}{y_j}\pdv{y_j}{x_i}\right)\geq \min_j \ord_D\left(\pdv{f}{y_j}\right) \text{ for all }i
\end{equation*}
and analogously in the other direction. This shows that the definition is independent on the choice of the system of parameters at $a$.

We now show that it is not dependent on choice of $a\in \sigma(D)$ either. Suppose that it is, so it can be considered as a function $k_D(a)$. Choose any $a'\in\sigma(D)$ together with a system of parameters $x_1,\dots,x_n$ at $a'$. If $A$ is any neighbourhood of $a'$ regular for this system then $x_1-x_1(a),\dots,x_n-x_n(a)$ make a system of parameters at every $a\in A$. This shows that $k_D$ is constant on $A\cap \sigma(D)$, so more generally it is locally constant on $\sigma(D)$ in the Zariski topology. Since the set is irreducible, $k_D$ is indeed a constant function of $a$.

We now recursively derive a useful bound on the value of the differential multiplicity:

\begin{prop}\label{recursive_k_mult}
Let $\sigma:X_\sigma\rightarrow X$ be a multiblowup and $\sigma_1:X_{\sigma_1}\rightarrow X_\sigma$ be a blowup with center $C$ inducing a divisor $D$ such that $\sigma\circ\sigma_1$ is admissible. Let $D_1,\dots,D_k$ be all the exceptional divisors of $\sigma$ containing $C$. Then
\begin{equation*}
    k_D\leq k_{D_1}+\dots+k_{D_k}+1
\end{equation*}
In particular if the collection $D_1,\dots,D_k$ is empty $k_D= 1$. 

Moreover, whenever $\dim(C)=n-k$ an even stronger bound holds:
\begin{equation*}
    k_D\leq k_{D_1}+\dots+k_{D_k}
\end{equation*}
\begin{proof}
To begin with notice that it follows from Proposition \ref{some_ineq_deriv} that the differential multiplicity is always positive. Choose a point $a_\sigma\in C$ lying on none other divisors than $D_1,\dots,D_k$. Let $u_1,\dots,u_n$ be a system of parameters at $a_\sigma$ with $u_i$ locally generating $D_i$ for $1\leq i\leq k$. Let $a:=\sigma(a_\sigma)$ and let $x_1,\dots,x_n$ be a system of parameters at $a$. By Observation \ref{diff_localmon} and the definition of differential multiplicity we have
\begin{equation}\label{inductive_bound_eq1}
    u_1^{k_{D_1}}\dots u_k^{k_{D_k}}\pdv{u_i}{x_j}\in\RR(X_\sigma)_{a_\sigma}\text{ for all }i,j.
\end{equation}
In particular
\begin{equation*}
    \ord_D\left(\pdv{u_i}{x_j}\right)\geq -k_{D_1}-\dots-k_{D_k}.
\end{equation*}
By the chain rule and Proposition \ref{val_diff_ineq} we get for all $f\in\R(X)$:
\begin{equation*}
    \ord_D\left(\pdv{f}{x_j}\right)\geq \min_i\ord_D\left(\pdv{f}{u_i}\right)+\ord_D\left(\pdv{u_i}{x_j}\right)\geq \ord_D(f)-1-k_{D_1}-\dots-k_{D_k}
\end{equation*}
which proves the first part.
In the case when $\dim(C)=n-k$ for every $i\leq k$ by the definition of differential multiplicity we have
\begin{equation*}
    \ord_{D_i}\left(\pdv{u_i}{x_j}\right)\geq \ord_{D_i}(u_i)-k_{D_i}= 1-k_{D_i}
\end{equation*}
so the equation (\ref{inductive_bound_eq1}) can be improved to give the following
\begin{align*}
    u_1^{k_{D_1}}\dots u_i^{k_{D_i}-1}\dots u_k^{k_{D_k}}\pdv{u_i}{x_j}&\in\RR(X_\sigma)_{a_\sigma}\text{ for all }1\leq i\leq k,1\leq j\leq n, \\
    \ord_D\left(\pdv{u_i}{x_j}\right)&\geq -k_{D_1}-\dots-k_{D_k}+1.
\end{align*}
This implies that if $i\leq k$ then
\begin{equation*}
    \ord_D\left(\pdv{f}{u_i}\right)+\ord_D\left(\pdv{u_i}{x_j}\right)\geq \ord_D(f)-k_{D_1}-\dots-k_{D_k}
\end{equation*}
If $i>k$ then the above equation stays true by the second part of Proposition \ref{val_diff_ineq}. The conclusion follows similarly to the first part.
\end{proof}
\end{prop}

\begin{rem}
The author has verified that the bounds above are optimal in the case $\dim(X)=2$ although they fail to do so in higher dimensions. It could be of some interest to find out precisely under what conditions does the stronger bound hold, as it would be a step towards extending the results of this paper outside the two-dimensional case.
\end{rem}

Let us now see an example:
\begin{ex}
Let $a\in X$ and let $D\subset X_\sigma$ be a divisor satysfying $\sigma(D)=\{a\}$. Let $x_1,\dots,x_n$ be a system of parameters at $a$. The differential multiplicity of $D$ satisfies the obvious inequality
\begin{equation*}
    k_D\geq \max(\ord_D(x_1),\dots,\ord_D(x_n)).
\end{equation*}
It may happen that the inequality becomes an equality, this is the case for example if $X=\R^n$ and the valuation $\ord_D$ is a monomial one. However, strict inequality may occur here as well. For example, with $X=\R^2$ there is a sequence of three blowups, which in a local chart is of the form
\begin{equation*}
    X_\sigma\ni (u,v)\mapsto ((uv+1)u^2,u)\in \R^2.
\end{equation*}
Taking $D$ as the exceptional divisor $u=0$ we get
\begin{align*}
    \ord_D(x)&=\ord_D((uv+1)u^2)=2,\\
    \ord_D(y)&=\ord_D(u)=1,\\
    \ord_D(x-y^2)&=\ord_D(u^3v)=3.
\end{align*}
Differentiating $x-y^2$ with respect to $x$ drops its valuation by $3$, showing that the differential multiplicity $k_D$ of $D$ is equal to at least $3$ (in this case using Proposition \ref{recursive_k_mult} one can show that it is precisely equal to $3$).
\end{ex}

The following definition plays a crucial role in the remainder of the paper:

\begin{defi}
Let $f\in \R(X)$. We say that $f$ is \emph{relatively $\RR^k$-flat} on $X_\sigma$ if $\ord_D(f)> kk_D$ holds for every exceptional divisor $D$ of $\sigma$ on $X_\sigma$.
\end{defi}

\begin{obs}\label{rel_flat_are_flat}
If $f\in\RR(X_\sigma)$ is relatively $\RR^k$-flat on $X_\sigma$ then it is $\RR^k$-flat (on $X$).
\begin{proof}
Take $a\in X$ and a system of parameters $x_1,\dots,x_n$ at $a$. Choose a representation $f=\frac{p}{q}$ as an irreducible fraction at $a$. By Proposition \ref{monomial_denom_representation} there exists an open neighbourhood $A$ of $a$ regular for the system $x_1,\dots,x_n$ such that $q$ is rsnc at every point of $\sigma^{-1}(A)$. This means that the set of valuations at the exceptional divisors of $\sigma\vert_{\sigma^{-1}(A)}$ contains all the divisorial valuations associated to $q\vert_A$. On the other hand by the definition of differential multiplicity for every such divisor $D$:
\begin{equation*}
    \ord_D(p)+ k\ord_D\left(\pdv{q}{x_i}\right)> kk_D+\ord_D(q)+k(\ord_D(q)-k_D)=(k+1)\ord_D(q)
\end{equation*}
so $\frac{p}{q}$ is an $\RR^k$-flat representation on $A$. 
\end{proof}
\end{obs}

Similarly we make an analogous definition:
\begin{defi}
Let $f\in \R(X)$. We say that $f$ is \emph{relatively $\Rr{k}$-flat} on $X_\sigma$ if $\ord_D(f)\geq kk_D$ holds for every exceptional divisor $D$ of $\sigma$ on $X_\sigma$.
\end{defi}

\begin{obs}
If $f\in\RR(X_\sigma)$ is relatively $\Rr{1}$-flat on $X_\sigma$ then it is $\Rr{1}$-flat (on $X$).
\begin{proof}
Similar to the proof of the last observation.
\end{proof}
\end{obs}

\subsection{Sums of squares}
We are ready to start some investigation on sums of squares in rings of regulous function. Firstly we require the following:
\begin{lem}\label{snc_is_smq}
Let $f\in\RR(X)$ be positive semi-definite and a simple normal crossing. Then it can be written as a sum of squares of elements from $\RR(X)$.
\begin{proof}
Locally at every point $a\in X$, $f$ can be written as 
\begin{equation*}
    f=u_1^{2\alpha_1}\dots u_n^{2\alpha_n} r
\end{equation*}
for some system of parameters $u_1,\dots,u_n$. As $r$ is a unit in $\RR(X)_a$. it must be a positive definite element of $\RR(X)_a$, so by the Positivestellensatz it is a sum of squares of elements of $\RR(X)_a$. Hence the same is true for $f$. Consider the set
\begin{equation*}
    J:=\{g\in\RR(X):fg^2\text{ is a sum of squares of elements of }\RR(X)\}.
\end{equation*}
We have just shown that for every $a\in X$ there exists $g\in J$ with $g(a)\neq 0$. By Noetherianity we can find a finite number of elements $g_1,\dots,g_k\in J$ with $\Z(g_1,\dots,g_k)=\emptyset$. Now $f(g_1^2+\dots+g_k^2)^2$ is a sum of squares in $\RR(X)$ and the conclusion follows.
\end{proof}
\end{lem}

\begin{cor}\label{squares_cor}
Let $f\in \RR(X)$ be psd. Suppose that there exists an admissible multiblowup $\sigma:X_\sigma\rightarrow X$ such that $f$ is snc and relatively $\RR^{2k}$-flat on $X_\sigma$. Then $f$ is a sum of squares of $\RR^k$-flat functions on $X$. Similarly if $f$ is relatively $\Rr{2}$-flat on $X_\sigma$, then it is a sum of squares of $\Rr{1}$-flat functions on $X$.
\begin{proof}
By Lemma \ref{snc_is_smq}, $f$ is a sum of squares in $\RR(X_\sigma)$
\begin{equation*}
    f=f_1^2+\dots+f_m^2.
\end{equation*}
Each of the functions $f_i$ is necessarily relatively $\RR^k$-flat on $X_\sigma$, hence by Proposition \ref{rel_flat_are_flat} $\RR^k$-flat on $X$. The second part follows similarly.
\end{proof}
\end{cor}

To prove Theorems \ref{sos1} and \ref{sos2} it now suffices to find an appropriate resolution of singularities making $f$ a simple normal crossing simultaneously keeping track of its orders at the exceptional divisors. It turns out that what we require can be found in \cite{resolution_snc}:

\begin{thm}[\cite{resolution_snc}]
Let $f\in\RR(X)$. There exists an admissible multiblowup $\sigma=\sigma_1\circ\dots\circ\sigma_t:X_\sigma\rightarrow X$ such that
\begin{enumerate}
    \item $f$ is a simple normal crossing on $X_\sigma$,
    \item for all $i$, $f$ is not a simple normal crossing at any of the points of the center of $\sigma_i$.
\end{enumerate}
\end{thm}

In this setting:

\begin{obs}\label{existence_of_resolution}
If $f$ is a sum of $2k$-th powers of elements of $\R(X)$ then it is relatively $\Rr{2k}$-flat on $X_\sigma$. If besides that every zero of $f$ is also a zero of its derivatives up to order $2k$ then it is relatively $\RR^{2k}$-flat on $X_\sigma$.
\begin{proof}
Consider the first part; we prove it inductively by showing that $\ord_D(f)\geq 2kk_D$ holds for every exceptional divisor of $\sigma_1\circ\dots\circ\sigma_j$. Suppose that the result holds for $j-1$, we have to show that it holds for the divisor $D$ induced by $\sigma_j$. Let $D_1,\dots,D_l$ be all the exceptional divisors containing the center of $\sigma_j$ and let $a_\sigma$ be a point of the center lying on none other divisors than these. Let $u_1,\dots, u_n$ be a regular system at $a_\sigma$ with $u_i$ locally generating $D_i$ (for $1\leq i\leq l$). By induction $f$ can locally be written as
\begin{equation}\label{proof_of_squares_eq1}
    f=u_1^{2kk_{D_1}}\dots u_l^{2kk_{D_l}}r.
\end{equation}
By the assumptions on the resolution $r$ is not a unit at $a_\sigma$, even further it is not a unit at any point of the center of $\sigma_j$ at which $u_1,\dots,u_n$ make a regular system of parameters. This means that $\ord_D(r)>0$. As $r$ is a sum of $2k$-th powers in $\R(X)$ we actually must have $\ord_D(r)\geq 2k$. Proposition \ref{recursive_k_mult} finishes the induction:
\begin{equation*}
    \ord_D(f)\geq 2kk_{D_1}+\dots+2kk_{D_l}+2k\geq 2kk_D.
\end{equation*}
For the second part we use a similar inductive argument, with a strict inequality $\ord_D(f)>2kk_D$ instead of a non-strict one. Since $f$ is a sum of $2k$-th powers the orders of $f$ at the divisors must be divisible by $2k$ so the equation (\ref{proof_of_squares_eq1}) now turns into
\begin{equation*}
    f=u_1^{2kk_{D_1}+2k}\dots u_l^{2kk_{D_l}+2k}r.
\end{equation*}
The induction is now finished in a similar manner:
\begin{equation*}
    \ord_D(f)\geq (2kk_{D_1}+2k)+\dots+(2kk_{D_l}+2k)+2k> 2kk_D.
\end{equation*}
This argument fails if $l=0$; we must then use the assumption on the derivative of $f$ to realise that the order of $f$ at the newly created divisor is greater than $2kk_D$ (as $k_D=1$ by Proposition \ref{recursive_k_mult}).
\end{proof}
\end{obs}

This together with Corollary \ref{squares_cor} immediately proves Theorems \ref{sos1} and \ref{sos2}.

\begin{rem}\label{codim2_existence}
We give another consequence of the existence of this strong resolution of singularities: if the zero set of $f$ is of codimension at least two then $f$ is relatively $\Rr{2}$-flat on $X_\sigma$. Indeed, for the above proof to work we just need to verify that the order of $f$ at any given divisor $D$ is even. This follows from the fact that $f$ is locally of constant sign on $X$, hence it is also locally of constant sign on $X_\sigma$ so it cannot switch signs when crossing $D$. This fact will be important in Section \ref{An additional result}.
\end{rem}

\section{$\RR^k$-flat decomposition in dimension 2}\label{RR^k-flat decomposition in dimension 2}
This is the most technical section of the paper, before jumping into let us first gain some motivation.

$\RR^k$-flat functions possess this nice property that it is quite easy to generate them; if we want construct an $\RR^1$-flat function with denominator equal to let us say $x^2+y^6$, then we just need to check what the divisorial valuations associated to it are (there is only the obvious monomial one, with differential multiplicity equal to three) and just like that we instantly know that all the functions below are $\RR^1$-flat:
\begin{equation*}
    \frac{x^3y}{x^2+y^6},\frac{x^2y^4}{x^2+y^6},\frac{xy^7}{x^2+y^6},\frac{y^{10}}{x^2+y^6}.
\end{equation*}
Now if we want to generate some less trivial $k$-regulous functions we can take the $\RR^k$-flat ones as basis and then consider the ring they span. Actually one can check easily that they are closed under multiplication, so it suffices to consider only finite sums. This way we can create really non-trivial examples like:
\begin{equation*}
    1+\frac{x^4+xy^7}{x^2+y^6}+\frac{x^6}{(x^2+y^2)(x^2+y^4)}=\dots\in\RR^1(\R^2)
\end{equation*}

Actually if one thinks about it, it is quite hard to come up with another way of generating $k$-regulous functions, possibly with the sole exception of inverting sums of the form $1+f_1^2+\dots+f_m^2$. Because of that, it seems reasonable to ask whether every $k$-regulous function can be written as a finite sum of $\RR^k$-flat functions. In this section we will answer this question affirmatively in dimension $2$. 

From now on we are only interested in the two dimensional case; we constantly assume that $\dim(X)=2$. In two dimensions every sequence of non-trivial blowups is admissible, so we might just speak of \emph{multiblowups} instead of \emph{admissible multiblowups}.

\begin{obs}\label{invert_chain_rule}
Let $\sigma:X_\sigma\rightarrow X$ be a multiblowup and let $a_\sigma\in X_\sigma$. Let $x,y$ and $u,v$ be two systems of parameters at $\sigma(a_\sigma)$ and $a_\sigma$ respectively. The following equality holds for $f\in\R(X)$:
\begin{equation*}
    \left(\pdv{x}{u}\pdv{y}{v}-\pdv{x}{v}\pdv{y}{u}\right)\pdv{f}{y}=\pdv{f}{v}\pdv{x}{u}-\pdv{f}{u}\pdv{x}{v}
\end{equation*}
\begin{proof}
Follows from the chain rule and the formula for matrix inverse.
\end{proof}
\end{obs}

By a \emph{Nash system of parameters at a point $a_\sigma\in X_\sigma$} we mean two Nash functions $u,v$ defined on a Euclidean neighbourhood of $a_\sigma$ in $X_\sigma$ such that $(u,v)$ is a local diffeomorphism at $a_\sigma$ taking $a_\sigma$ to $(0,0)$.

\begin{prop}\label{matrix_det_ineq}
Let $\sigma:X_\sigma\rightarrow X$ be a multiblowup and let $a_\sigma\in X_\sigma$ lie on an exceptional divisor $D$ of $\sigma$. Let $a:=\sigma(a_\sigma)$ and let $x,y$ be a system of parameters at $a$. Let $u,v$ be a Nash system of parameters at $a_\sigma$, with $\Z(u)$ locally equal to $D$ on a neighbourhood of $a_\sigma$. Define 
\begin{equation*}
    \gamma_1:=k_D+\min(\ord_D(x),\ord_D(y))-1
\end{equation*}
The following function admits a Nash extension from its natural domain to a Euclidean neighbourhood of $a_\sigma$:
\begin{equation*}
    \frac{\pdv{x}{u}\pdv{y}{v}-\pdv{x}{v}\pdv{y}{u}}{u^{\gamma_1}}
\end{equation*}

Moreover, if $a$ lies on intersection of $D$ with another exceptional divisor $G$, and $\Z(v)$ locally equals $G$ then after defining
\begin{equation*}
    \gamma_2:=k_G+\min(\ord_G(x),\ord_G(y))-1    
\end{equation*}
the claim is also true for 
\begin{equation*}
    \frac{\pdv{x}{u}\pdv{y}{v}-\pdv{x}{v}\pdv{y}{u}}{u^{\gamma_1} v^{\gamma_2}}
\end{equation*}
\begin{proof}
We firstly consider the case when $u,v$ are algebraic, i.e. they belong to $\RR(X_\sigma)_{a_\sigma}$. Take $f$ to be a function attaining the bound of the differential multiplicity of $D$, i.e. a rational function satisfying
\begin{equation*}
    k_D=\ord_D(f)-\min\left(\ord_D\left(\pdv{f}{x}\right),\ord_D\left(\pdv{f}{y}\right)\right)
\end{equation*}
Without loss of generality $\min(\ord_D(\pdv{f}{x}),\ord_D(\pdv{f}{y}))=\ord_D(\pdv{f}{y})$. By Proposition \ref{val_diff_ineq}
\begin{align*}
    \ord_D\left(\pdv{f}{v}\pdv{x}{u}\right)&\geq \ord_D(f)+\ord_D(x)-1,\\
    \ord_D\left(\pdv{f}{u}\pdv{x}{v}\right)&\geq \ord_D(f)+\ord_D(x)-1.
\end{align*}
Hence by Observation \ref{invert_chain_rule}
\begin{equation*}
    \ord_D\left(\pdv{x}{u}\pdv{y}{v}-\pdv{x}{v}\pdv{y}{u}\right)\geq \ord_D(f)-\ord_D\left(\pdv{f}{y}\right)+\ord_D(x)-1\geq \gamma_1.
\end{equation*}
Since $u$ locally generates $D$ and $\pdv{x}{u}\pdv{y}{v}-\pdv{x}{v}\pdv{y}{u}\in\RR(X_\sigma)_{a_\sigma}$ the conclusion follows. The second part also follows by this argument since we analogously have
\begin{equation*}
    \ord_G\left(\pdv{x}{u}\pdv{y}{v}-\pdv{x}{v}\pdv{y}{u}\right)\geq\gamma_2.
\end{equation*}

In the general case, if $u,v$ is a Nash system of parameters then choose $u',v'$ to be an algebraic one with $u'$ locally generating $D$ (and $v'$ locally generating $G$ when proving the second part). The formula for a change of coordinates gives us
\begin{equation*}
    \begin{bmatrix}
        \pdv{x}{u} & \pdv{x}{v} \\
        \pdv{y}{u} & \pdv{x}{v}
    \end{bmatrix}
    =
    \begin{bmatrix}
        \pdv{x}{u'} & \pdv{x}{v'} \\
        \pdv{y}{u'} & \pdv{x}{v'}
    \end{bmatrix}
    \begin{bmatrix}
        \pdv{u'}{u} & \pdv{u'}{v} \\
        \pdv{v'}{u} & \pdv{v'}{v}
    \end{bmatrix}
\end{equation*}
Taking determinants of these matrices finishes the proof.
\end{proof}
\end{prop}

We will have to prove three bounds on regulous functions satisfying certain valuation inequalities. We firstly state these results and then move onto the proofs.

\begin{prop}\label{first_anal_prop}
Let $a_\sigma\in X_\sigma$ lie on precisely one exceptional divisor $D$, locally generated by $u\in\RR(X_\sigma)_{a_\sigma}$. Let $f\in\RR^k(X)$ satisfy $\ord_D(f)> kk_D$. Then
\begin{equation*}
    v_{u,U}(f)>kk_D
\end{equation*}
holds for some Euclidean neighbourhood $U$ of $a_\sigma$.
\end{prop}

\begin{prop}\label{second_anal_prop}
Let $a_\sigma\in X_\sigma$ lie on precisely one exceptional divisor $D$, locally generated by $u\in\RR(X_\sigma)_{a_\sigma}$. Let $v\in\RR(X_\sigma)_{a_\sigma}$ be such that $u,v$ make a system of parameters at $a_\sigma$. Denote by $G$ the germ of $\Z(v)$. Let $0\leq l\leq k$ and let $f\in\RR^k(X)$ satisfy $\ord_D(f)> k k_D$ and $\ord_G(f)\geq l$. Then
\begin{equation*}
    v_{u,U}\left(\frac{f}{v^l}\right)>kk_D
\end{equation*}
holds for some Euclidean neighbourhood $U$ of $a_\sigma$ (in particular $\frac{f}{v^l}$ is locally bounded).
\end{prop}

\begin{prop}\label{third_anal_prop}
Let $a_\sigma\in X_\sigma$ lie on two exceptional divisors $D$ and $G$, locally generated by $u,v\in\RR(X_\sigma)_{a_\sigma}$ respectively. Let $f\in\RR^k(X)$ satisfy $\ord_D(f)>k k_D$, $\ord_G(f)>k k_G$. Then
\begin{equation*}
    v_{u^{k_D}v^{k_G},U}\left(f\right)>k
\end{equation*}
holds for some Euclidean neighbourhood $U$ of $a_\sigma$.
\end{prop}

\begin{proof}[Proof of Proposition \ref{first_anal_prop}]
We use induction on $k$ the case $k=0$ following from Łojasiewicz's inequality (\cite[Proposition 2.6.4]{Bochnak}). Choose a system of parameters $x,y$ at $a$ with $(x,y)=(x)$ in $\RR(X_\sigma)_{a_\sigma}$, and let $v\in\RR(X_\sigma)_{a_\sigma}$ be such that $u,v$ make a system of parameters at $a_\sigma$. It follows from Proposition \ref{rsnc_stays_after_blowup} that $x$ is rsnc at $a_\sigma$, i.e.
\begin{equation*}
    x=u^{\ord_D(x)}r
\end{equation*}
where $r\in\RR(X_\sigma)_{a_\sigma}^\ast$. Without loss of generality we may assume that $r(a_\sigma)>0$. We define a new Nash system of parameters at $a_\sigma$:
\begin{align*}
    u':&=ur^{\frac{1}{\ord_D(x)}}\\
    v':&=v
\end{align*}
The fact that it is a system of parameters can be directly checked by computation of its Jacobian matrix with respect to the system $u,v$. Let $U$ be a small neighbourhood of $a_\sigma$ mapped diffeomorphically by $(u',v')$ onto an open subset of $\R^2$.

The Jacobian matrix of $(x,y)$ with respect to the new system looks as follows:
\begin{equation*}
    \begin{bmatrix}
       \ord_D(x){u'}^{\ord_D(x)-1} & 0\\
       \pdv{y}{u'} & \pdv{y}{v'}
    \end{bmatrix}
\end{equation*}
Using Proposition \ref{matrix_det_ineq} we get:
\begin{align*}
    v_{u,U}\left(\ord_D(x){u'}^{\ord_D(x)-1}\pdv{y}{v'}\right)&\geq k_D+\ord_D(x)-1\\
    v_{u,U}\left(\pdv{y}{v'}\right)&\geq k_D.
\end{align*}
On the other hand (after decreasing $U$ if necessary) we can apply the inductive hypothesis to $\pdv{f}{y}$ to get
\begin{equation*}
    v_{u,U}\left(\pdv{f}{y}\right)>(k-1)k_D.
\end{equation*}
Thus by the chain rule
\begin{align}
    \pdv{f}{v'}=\pdv{f}{x}\pdv{x}{v'}+\pdv{f}{y}\pdv{y}{v'}&=\pdv{f}{y}\pdv{y}{v'}\nonumber\\
    v_{u,U}\left(\pdv{f}{v'}\right)&>kk_D\nonumber\\
    v_{u',U}\left(\pdv{f}{v'}\right)&=kk_D+\epsilon.\label{first_anal_prop_eq1}
\end{align}
for some $\epsilon>0$.
We can assume that in the coordinates $(u',v')$ $U$ is a square, i.e. it is of the form $\{(u',v'):|u'|<\delta,|v'|<\delta\}$. Let us fix some $0<v_0<\delta$. Then (\ref{first_anal_prop_eq1}) together with the mean value theorem gives us:
\begin{equation*}
    |f(u',v')|\leq C|u'|^{kk_D+\epsilon}+|f(u',v_0)|
\end{equation*}
for a constant $C$ independent of the choice of $(u',v')\in U$. Since $\ord_D(f)>kk_D$, we can choose $v_0$ to be such that $\frac{f}{u^{kk_D+1}}$ is regular at $(0,v_0)$; this implies that
\begin{equation*}
    |f(u',v_0)|\leq C_1|u'|^{kk_D+1}
\end{equation*}
and the conclusion follows.
\end{proof}

\begin{proof}[Proof of Proposition \ref{second_anal_prop}]
All the reasoning is still conducted in the system of parameters $u',v'$ from the last proof. We again proceed by induction on $k$. If $l=0$ then the conclusion is equivalent to Proposition \ref{first_anal_prop}, so we assume $l>0$. 

To begin with recall the equation 
\begin{equation*}
    \left(\pdv{x}{u}\pdv{y}{v}-\pdv{x}{v}\pdv{y}{u}\right)\pdv{f}{y}=\pdv{f}{v}\pdv{x}{u}-\pdv{f}{u}\pdv{x}{v}
\end{equation*}
from Observation \ref{invert_chain_rule}.
The determinant $\pdv{x}{u}\pdv{y}{v}-\pdv{x}{v}\pdv{y}{u}$ is non-zero at points where $\sigma$ is an isomorphism, so 
\begin{equation*}
    \ord_G\left(\pdv{x}{u}\pdv{y}{v}-\pdv{x}{v}\pdv{y}{u}\right)=0.
\end{equation*}
Using Proposition \ref{val_diff_ineq} we obtain
\begin{equation*}
    \ord_G\left(\pdv{f}{y}\right)\geq l-1.
\end{equation*}
Furthermore by the definition of differential multiplicity $\ord_D(\pdv{f}{y})> (k-1)k_D$ so we can apply the inductive hypothesis to $\pdv{f}{y}$ to get
\begin{equation*}
    v_{u,U}\left(\frac{\pdv{f}{y}}{v^{l-1}}\right)>(k-1)k_D
\end{equation*}
for some small neighbourhood $U$ of $a_\sigma$.
From the equations $\pdv{f}{v'}=\pdv{f}{y}\pdv{y}{v'}$ and $v_{u,U}(\pdv{y}{v'})\geq k_D$ we derived in the preceding proof we deduce
\begin{equation*}
    v_{u',U}\left(\frac{\pdv{f}{v'}}{v'^{l-1}}\right)>kk_D.
\end{equation*}
In other words, decreasing $U$ if necessary, there exist constants $C,\epsilon>0$ such that
\begin{equation*}
    \left|\pdv{f}{v'}\right|\leq C|v'|^{l-1}|u'|^{kk_D+\epsilon}
\end{equation*}
holds in $U$. We can now apply the mean value theorem to get the desired result, since $f(u',0)=0$.
\end{proof}
\setcounter{equation}{0}

\begin{proof}[Proof of Proposition \ref{third_anal_prop}]
Again we use induction on $k$, the case $k=0$ being trivial. Choose a system of parameters $x,y$ at $a$ with $(x,y)=(x)$ in $\RR(X_\sigma)_{a_\sigma}$. Again $x$ is rsnc at $a_\sigma$, so it can be written as
\begin{equation*}
    x=u^\alpha v^\beta r
\end{equation*}
where $\alpha=\ord_D(x)>0$, $\beta=\ord_G(x)>0$, $r\in\RR(X_\sigma)_{a_\sigma}^\ast$. Without loss of generality we assume $r(a_\sigma)>0$ and 
\begin{equation*}
    \beta k_D<\alpha k_G \text{ or there is an equality and }\beta\leq\alpha.
\end{equation*}
We define a new Nash system of parameters at $a_\sigma$ this time by
\begin{align*}
    u'&:=ur^{\frac{1}{\alpha}}\\
    v'&:=v
\end{align*}
The fact that it is a system of coordinates again follows from simple calculation of its Jacobian matrix. In these coordinates $x=u'^\alpha v'^\beta$. Recall the equation from Observation \ref{invert_chain_rule}: 
\begin{equation*}
    \left(\pdv{x}{u'}\pdv{y}{v'}-\pdv{x}{v'}\pdv{y}{u'}\right)\pdv{f}{y}=\pdv{f}{v'}\pdv{x}{u'}-\pdv{f}{u'}\pdv{x}{v'}
\end{equation*}
By the inductive hypothesis applied to $\pdv{f}{y}$ together with Proposition \ref{matrix_det_ineq} this implies
\begin{equation*}
    v_{uv,U}\left(\frac{\pdv{f}{v'}\pdv{x}{u'}-\pdv{f}{u'}\pdv{x}{v'}}{u^{kk_D+\alpha-1}v^{kk_G+\beta-1}}\right)>0
\end{equation*}
for some small $U$. In other words we can find positive constants $\epsilon$ and $C$ such that
\begin{equation}\label{third_anal_eq1}
    \left|\pdv{f}{v'}\pdv{x}{u'}-\pdv{f}{u'}\pdv{x}{v'}\right|\leq C|u'|^{kk_D+\alpha-1+\epsilon}|v'|^{kk_G+\beta-1+\epsilon}
\end{equation}
We again assume that $U$ is mapped diffeomorphically by $(u',v')$ onto a square. Choose some small $\gamma>1$ and a point $(u_0,v_0)\in U$ (in the coordinates $u',v'$). We define an arc $\xi:[1,\gamma]\rightarrow U$ by 
\begin{equation*}
    \xi:t\mapsto (u_0t^\beta,v_0t^{-\alpha})
\end{equation*}
We would like to use the fundamental theorem of calculus to bound the rate of change of $f$ along the arc, to do that we need to make some preliminary computations: 
\begin{align}
    \pdv{x}{u'}\circ\xi(t)&=\alpha u_0^{\alpha-1} v_0^\beta t^{-\beta}\nonumber \\
    \pdv{x}{v'}\circ\xi(t)&=\beta u_0^\alpha v_0^{\beta-1} t^{\alpha}\nonumber \\
    \frac{d\xi}{dt}=[\beta u_0t^{\beta-1},-\alpha v_0t^{-\alpha-1}]&=\left[\pdv{x}{v'}\circ\xi(t),-\pdv{x}{u'}\circ\xi(t)\right]u_0^{1-\alpha}v_0^{1-\beta}t^{\beta-\alpha-1} \label{third_anal_eq2}
\end{align}
Now we use the fundamental theorem of calculus and equations (\ref{third_anal_eq1}) and (\ref{third_anal_eq2}) to get
\begin{multline*}
    |f(u_0\gamma^\beta,v_0\gamma^{-\alpha})-f(u_0,v_0)|=\left|\int_1^\gamma df(\xi(t))\cdot d\xi(t)\; dt\right|=\\
    =\left|\int_1^\gamma
    \left(\pdv{f}{u'}\pdv{x}{v'}\circ\xi(t)-\pdv{f}{v'}\pdv{x}{u'}\circ\xi(t)\right) u_0^{1-\alpha}v_0^{1-\beta}t^{\beta-\alpha-1}dt\right|\leq\\
    \leq\left|Cu_0^{kk_D+\epsilon}v_0^{kk_G+\epsilon}\int_1^\gamma t^{\beta kk_D-\alpha kk_G+\epsilon\beta-\epsilon\alpha-1}dt\right|\leq
    \left|C_1 u_0^{kk_D+\epsilon}v_0^{kk_G+\epsilon}\log(\gamma)\right|
\end{multline*}
The last inequality stems from the fact that by our assumptions the exponent in the integrand is less than or equal to $-1$ if $\epsilon$ is small enough (here $C_1$ is independent of $u_0,v_0$ and $\gamma$).

Now let us now fix a generic value $u_1$ such that $\frac{f}{v^{kk_G+1}}$ is regular at $(u_1,0)$. If $(u_0,v_0)$ is sufficiently close to $(0,0)$ and $u_0$ is of the same sign as $u_1$ then after setting $\gamma=(\frac{u_1}{u_0})^\frac{1}{\beta}$ we have that $\xi$ is a well defined map into $U$. Hence we obtain:
\begin{multline*}
    \left|\frac{f(u_0,v_0)}{u_0^{kk_D} v_0^{kk_G}}\right|\leq |C_1u_0^{\epsilon}v_0^{\epsilon}\log(\gamma)|+\left|\frac{f(u_1,v_0(\frac{u_2}{u_0})^{-\frac{\alpha}{\beta}})}{u_0^{kk_D}v_0^{kk_G}}\right|\leq \\
    \leq |C_2u_0^{\epsilon/2}v_0^{\epsilon}|+\left|C_3v_0 u_0^{\frac{\alpha}{\beta}kk_G-kk_D+\frac{\alpha}{\beta}}\right|
\end{multline*}
which finishes the proof.
\end{proof}

One important corollary of these results is the following:
\begin{prop}\label{mult_stays_kregulous}
Let $f\in\RR^k(X)$ be relatively $\RR^k$-flat on $X_\sigma$. Let $g\in\RR(X_\sigma)$. Then $fg\in\RR^k(X)$.
\begin{proof}
The k-th partial derivatives of $fg$ are all sums of functions in the form
\begin{equation*}
    \frac{\partial^{k-i}f}{\partial x^{j_1} \partial y^{j_2}}\frac{\partial^i g}{\partial x^{i_1} \partial y^{i_2}}
\end{equation*}
for $0\leq i \leq k$, $i_1+i_2=i$, $j_1+j_2=k-i$. We prove that all such functions approach zero at points of the exceptional locus of $\sigma$, so they are $0$-regulous on $X$. Fix a point $a_\sigma$ lying in the exceptional locus. We distinguish two cases:
\begin{case}
$a_\sigma$ lies on precisely one exceptional divisor $D$.
\begin{proof}
Take $u,v$ to be a system of parameters at $a_\sigma$ with $u$ locally generating $D$.
By Observation \ref{diff_localmon} and the definition of differential multiplicity we have $u^{ik_D}\frac{\partial^i g}{\partial^{i_1} x \partial^{i_2} y}\in\RR(X_\sigma)_{a_\sigma}$. On the other hand by Proposition $\ref{first_anal_prop}$
\begin{equation*}
    v_{u,U}\left(\frac{\partial^{k-i}f}{\partial x^{j_1} \partial y^{j_2}}\right)>ik_D.
\end{equation*}
Hence
\begin{equation*}
     \frac{\partial^{k-i}f}{\partial x^{j_1} \partial y^{j_2}}\frac{\partial^i g}{\partial x^{i_1} \partial y^{i_2}}=\frac{\frac{\partial^{k-i}f}{\partial x^{j_1} \partial y^{j_2}}}{u^{ik_D}}u^{ik_D}\frac{\partial^i g}{\partial x^{i_1} \partial y^{i_2}}\text{ approaches zero at }a_\sigma.
\end{equation*}
\end{proof}
\end{case}

\begin{case}
$a_\sigma$ lies on two exceptional divisors $D$ and $G$.
\begin{proof}
Take $u,v$ to be a system of parameters locally generating $D,G$ respectively.
Again Observation \ref{diff_localmon} together with the definition of differential multiplicity gives us $u^{ik_D}v^{ik_G}\frac{\partial^i g}{\partial x^{i_1} \partial y^{i_2}}\in\RR(X_\sigma)_{a_\sigma}$.
On the other hand by Proposition \ref{third_anal_prop}
\begin{equation*}
    v_{uv,U}\left(\frac{\frac{\partial^{k-i}f}{\partial x^{j_1} \partial y^{j_2}}}{u^{ik_D}v^{ik_G}}\right)>0.
\end{equation*}
Hence
\begin{equation*}
     \frac{\partial^{k-i}f}{\partial x^{j_1} \partial y^{j_2}}\frac{\partial^i g}{\partial x^{i_1} \partial y^{i_2}}=
     \frac{\frac{\partial^{k-i}f}{\partial x^{j_1} \partial y^{j_2}}}{u^{ik_D}v^{ik_G}}u^{ik_D}v^{ik_G}\frac{\partial^i g}{\partial x^{i_1} \partial y^{i_2}}\text{ approaches zero at }a_\sigma.
\end{equation*}

\end{proof}
\end{case}
\end{proof}
\end{prop}

We are now almost ready to present the proof of the $\RR^k$-flat decomposition. It is kind of a recursive algorithm which repeatedly applies the following proposition:
\begin{prop}
Let $f\in\RR^k(X)$ be relatively $\RR^k$-flat on $X_\sigma$ and let $\sigma_1:X_{\sigma_1} \rightarrow X_\sigma$ be the blowup at $a_\sigma\in X_\sigma$. Then $f$ can be written as
\begin{equation*}
    fg=f_1+f_2
\end{equation*}
where 
\begin{itemize}
    \item $g\in\RR(X_\sigma)$, $g(a_\sigma)\neq 0$,
    \item $f_1\in\RR(X_\sigma)$ is relatively $\RR^k$-flat on $X_\sigma$,
    \item $f_2\in \R(X)$ is relatively $\RR^k$-flat on $X_{\sigma_1}$.
\end{itemize}
\begin{proof}
\begin{case}
$a_\sigma$ does not lie in the exceptional locus of $\sigma$. 
\begin{proof}
We define $f_1$ to be $f$'s k-th Taylor polynomial at $a_\sigma$, i.e.
\begin{align*}
    f_1&:=\sum_{i+j\leq k}\frac{1}{i!j!}x^iy^j\frac{\partial^{i+j} f}{\partial x^i\partial y^j}(a_\sigma)\\
    f_2&:=f-f_1
\end{align*}
Using Taylor's theorem it is clear that the order of $f_2$ at the newly created divisor is greater than $kk_D$ (as $k_D=1$ by Proposition \ref{recursive_k_mult}). It suffices to multiply the second equation firstly by the denominators of $f_1$ and $f_2$, and then by a high enough power of a regular function vanishing on the exceptional locus of $\sigma$.
\end{proof}
\end{case}

\begin{case}
$a_\sigma$ lies on two exceptional divisors $D$ and $G$.
\begin{proof}
In this case Proposition \ref{third_anal_prop} tells us that the order of $f$ at the newly created divisor is greater than $kk_D+kk_G$, so by Proposition \ref{recursive_k_mult} (the "moreover" part) we can just take $g=1,f_1=0,f_2=f$. 
\end{proof}
\end{case}

\begin{case}
$a_\sigma$ lies on precisely one exceptional divisor $D$. 
\begin{proof}
This is the most involved case. Let $u,v$ be a system of parameters at $a_\sigma$ with $u$ locally generating $D$. Define $G$ as the germ of $\Z(v)$ at $a_\sigma$.
\begin{claim*}
For every $0\leq l \leq k$ there exist $g\in\RR(X_\sigma)$, $g(a_\sigma)\neq 0$ and $f_1\in\RR(X_\sigma)$, $f_1$ relatively $\RR^k$-flat on $X_\sigma$ such that $\ord_{G}(gf-f_1)\geq l$.
\begin{proof}
We reason inductively; when $l=0$ we take $g:=1,f_1:=0$ and there is nothing to prove.

Apply the inductive hypothesis to find $g,f_1$ with
\begin{equation*}
    \ord_{G}(gf-f_1)\geq l-1.
\end{equation*}
The function in the above equation is $k$-regulous because $f_1$ satisfies the assumptions of Proposition \ref{rel_flat_are_flat}, and $f$ and $g$ satisfy the assumptions of Proposition \ref{mult_stays_kregulous}. Hence Proposition \ref{second_anal_prop} applies here and gives
\begin{equation}\label{inductive_eq1}
    v_{u,U}\left(\frac{gf-f_1}{v^{l-1}}\right)>kk_D.
\end{equation}
Now the function 
\begin{equation*}
    \frac{gf-f_1}{u^{kk_D+1}v^{l-1}}    
\end{equation*}
belongs to $\RR(X_\sigma)_G$ so it makes sense to talk about its restriction to $G$. Because of (\ref{inductive_eq1}) the restriction is regular at $a_\sigma$, so we can extend it to a function $h\in \RR(X_\sigma)_{a_\sigma}$.
By the definition
\begin{equation*}
    \ord_G(gf-f_1-u^{kk_D+1}v^{l-1}h)\geq l.
\end{equation*}
Now it suffices to multiply the function in the above equation by some $g_1\in\RR(X_\sigma), g_1(a_\sigma)\neq 0$ vanishing in high order on all the other exceptional divisors to get rid of the denominators so that 
\begin{align*}
    g'&=gg_1\\
    f_1'&=g_1(f_1+u^{kk_D+1}v^lh)
\end{align*}
satisfy the desired conditions.
\end{proof}
\end{claim*}

In the end we apply the claim with $l=k$ and set $f_2:=gf-f_1$. As before Proposition \ref{second_anal_prop}  tells us that 
\begin{equation*}
    v_{u,U}\left(\frac{f_2}{v^{k}}\right)>kk_D.
\end{equation*}
This implies that the order of $f_2$ at the newly created divisor is greater than $kk_D+k$, so by Proposition \ref{recursive_k_mult} this is the decomposition we seek.
\end{proof}
\end{case}
\end{proof}
\end{prop}

Actually we can easily get rid off of the denominator $g$.
\begin{cor}\label{induction_decomp}
Let $f\in\RR^k(X)$ be relatively $\RR^k$-flat on $X_\sigma$ and let $\sigma_1:X_{\sigma_1} \rightarrow X_\sigma$ be the blowup at $a_\sigma\in X_\sigma$. Then $f$ can be written as
\begin{equation*}
    f=f_1+f_2
\end{equation*}
where $f_1\in\RR(X_\sigma)$ is relatively $\RR^k$-flat on $X_\sigma$ and $f_2\in \R(X)$ is relatively $\RR^k$-flat on $X_{\sigma_1}$.
\begin{proof}
Take $g,f_1,f_2$ from the last proposition. Let $h\in\RR(X_\sigma)$ be any function with zero set equal to $\{a_\sigma\}$ and let $m$ denote the differential multiplicity of the divisor induced by $\sigma_1$. Then
\begin{equation*}
    f=\frac{gf_1}{g^2+h^{2km}}+\frac{gf_2+fh^{2km}}{g^2+h^{2km}}
\end{equation*}
is the desired decomposition.
\end{proof}
\end{cor}

\begin{thm}[$\RR^k$-flat decomposition in dimension 2]\label{rk_flat_decomposition}
Let $f\in\RR^k(X)$. Then $f$ can be written as a sum of $\RR^k$-flat functions on $X$.
\begin{proof}
Take a sequence of blowups after which $f$ becomes regular. 
\begin{center}
    \begin{tikzcd}
    X & X_1 \arrow[l, "\sigma_1"'] & X_2 \arrow[l, "\sigma_2"'] & \dots \arrow[l, "\sigma_3"'] & X_n \arrow[l, "\sigma_n"']
    \end{tikzcd}
\end{center}
By Corollary \ref{induction_decomp} $f$ can be written as 
\begin{equation*}
    f=g_1+f_1
\end{equation*}
where $g_1$ is regular (and hence $\RR^k$-flat) on $X$, meanwhile $f_1$ is relatively $\RR^k$-flat on $X_1$.
Since $g_1$ is regular, $f_1$ is $k$-regulous. We continue and again use Corollary \ref{induction_decomp} to find a decomposition
\begin{equation*}
    f_1=g_2+f_2
\end{equation*}
such that $g_2$ is regular and relatively $\RR^k$-flat on $X_1$, meanwhile $f_2$ is relatively $\RR^k$-flat on $X_2$. Now $g_2$ is $\RR^k$-flat on $X$ by Proposition \ref{rel_flat_are_flat}, so both $g_2$ and $f_2$ are $k$-regulous.

Now we write down $f_2$ as 
\begin{equation*}
    f_2=g_3+f_3
\end{equation*}
with $g_3$ regular and relatively $\RR^k$-flat on $X_2$ and $f_3$ relatively $\RR^k$-flat on $X_3$. We repeat the reasoning until we get to $X_n$; then $f_n$ is relatively $\RR^k$-flat and regular on $X_n$ hence it is $\RR^k$-flat. This way we obtain the required decomposition
\begin{equation*}
    f=g_1+g_2+\dots+g_n+f_n.
\end{equation*}
\end{proof}
\end{thm}

\begin{rem}\label{decomp_rem}
The algorithm above gives something a bit stronger, given a resolution $\sigma=\sigma_1\circ\dots\sigma_n$ of $f$ it decomposes $f$ into a sum of functions $\Sigma f_i$ with $f_i$ being regular and relatively $\RR^k$-flat on $X_i$ ($X_0:=X$). This fact will be useful in the next section.
\end{rem}

\section{Proof of Theorem \ref{additional_theorem}}\label{An additional result}
\begin{lem}\label{polynomial_form_lemma}
Let $m>l>0$ and let $a_1,\dots,a_n$ be variables. Then $(a_1+\dots+a_n)^m$ can be written as a linear combination (with coefficients in $\mathbb{Z}$) of polynomials in the forms
\begin{equation}\label{form_of_the_polynomial}
    ba_i^{l+1-\alpha}(a_{i+1}+\dots+a_n)^\alpha
\end{equation}
for $1\leq i\leq n$, $0\leq \alpha\leq l$ and $b$ being a monomial (of order $m-l-1$) in the variables $a_1,\dots,a_i$.
\begin{proof}
Expanding $(a_1+\dots+a_n)^m$ in the first variable it gets represented as a sum of polynomials of the forms
\begin{equation*}
    a_1^{m-\alpha}(a_2+\dots+a_n)^\alpha
\end{equation*}
If $\alpha\leq l$ then the polynomial is of the form (\ref{form_of_the_polynomial}). Otherwise we use induction to represent $(a_2+\dots+a_n)^\alpha$ as a linear combination of polynomials of the desired form and notice that multiplication by $a_1^{m-\alpha}$ makes them stay that way.
\end{proof}
\end{lem}

We will now present the proof of Theorem \ref{additional_theorem}. For clarity let us first recall its statement:
\begin{thmABC}{\ref{additional_theorem}}
Let $m,l\geq 0$. Let $\dim(X)=2$ and let $f\in\RR^k(X)$ be written as a fraction $f=\frac{p}{q}$ for $p,q\in\RR(X)$. Then
\begin{equation*}
    p^lf^m\in\RR^{kl+k+2l}(X).
\end{equation*}
\begin{proof}
Notice that $p^lf^m=q^lf^{m+l}$ so after substituting $m:=m+l$ we deal with $q^lf^m$ instead. We first consider the case when the zero set of $q$ is of codimension at least 2. Using Remark \ref{codim2_existence} we find a multiblowup
\begin{center}
    \begin{tikzcd}
    X & X_1 \arrow[l, "\sigma_1"'] & X_2 \arrow[l, "\sigma_2"'] & \dots \arrow[l, "\sigma_3"'] & X_n \arrow[l, "\sigma_n"']
    \end{tikzcd}
\end{center}
such that $q$ is snc and relatively $\Rr{2}$-flat on $X_n$. By the $\RR^k$-flat decomposition (Remark \ref{decomp_rem}) we can write $f$ as
\begin{equation*}
    f=f_0+\dots+f_n
\end{equation*}
with $f_i\in\RR(X_i)$ being relatively $\RR^k$-flat on $X_i$. Ignoring the trivial cases we may assume that $m>l>0$, we can now apply Lemma \ref{polynomial_form_lemma} to write $q^lf^m$ as a linear combination of functions in the forms
\begin{equation*}
    q^lbf_i^{l+1-\alpha}(f_{i+1}+\dots+f_n)^\alpha
\end{equation*}
for $0\leq i\leq n$, $0\leq \alpha\leq l$ and $b$ being a monomial in the variables $f_0,\dots,f_i$. We are to show that all these functions are $kl+k+2l$-regulous. Notice that 
\begin{equation*}
    q(f_{i+1}+\dots+f_n)=p-q(f_0+\dots+f_i)\in\RR(X_i)
\end{equation*}
so
\begin{equation*}
    q^lbf_i^{l+1-\alpha}(f_{i+1}+\dots+f_n)^\alpha=q^{l-\alpha}bf_i^{l+1-\alpha}(q(f_{i+1}+\dots+f_n))^\alpha\in\RR(X_i)
\end{equation*}
We claim that the function above is relatively $\RR^{kl+k+2l}$-flat on $X_i$. Both $f_i$ and $f_{i+1}+\dots+f_n$ are relatively $\RR^k$-flat on $X_i$ meanwhile $q$ is relatively $\Rr{2}$-flat on $X_i$. This means that for every exceptional divisor $D$ on $X_i$:
\begin{equation*}
    \ord_D(q^lbf_i^{l+1-\alpha}(f_{i+1}+\dots+f_n)^\alpha)> 2lk_D+(l+1)kk_D=(kl+k+2l)k_D
\end{equation*}
and the claim follows from Proposition \ref{rel_flat_are_flat}.

In the general case, choose a point $a\in X$ and write $f=\frac{p_1}{q_1}$ as an irreducible fraction in $\RR(X)_a$. It is well known that the indeterminancy locus of $f$ is of codimension at least $2$, so the zero set of $q_1$ is of codimension at least $2$ when restricted to a small Zariski neighbourhood $A$ of $a$. Making it even smaller we may assume that $p_1,\frac{q}{q_1}\in\RR(A)$. By the already covered case:
\begin{equation*}
    q^lf^m=\left(\frac{q}{q_1}\right)^lq_1^lf^m\in\RR^{kl+k+2l}(A).
\end{equation*}
Since $a$ was arbitrary the proof is finished.
\end{proof}
\end{thmABC}
\begin{rem}
The theorem is optimal with $X=\R^2$ as long as $m>0$. For $l=0$ that is obvious, for $m,l>0$ an explicit example can be given by
\begin{equation*}
    f:=1+\frac{y^{3kl+6l+1}}{y^{6l}+x^2}
\end{equation*}
The fact that $f$ is $k$-regulous can be checked directly by verifying that $f-1$ is $\RR^k$-flat. Checking that $p^lf^m$ with $p=y^{3kl+6l+1}+y^{6l}+x^2$ is not $kl+k+2l+1$-regulous is a bit more difficult, although it can be done in an elementary way. We omit the details.
\end{rem}

\section{Discussion}\label{Discussion}
The main result of the paper applies only in the two-dimensional case, so it leaves the following question still unanswered: 
\begin{que}
    Can every $k$-regulous function defined on a variety of dimension greater than 2 be extended to an ambient variety?
\end{que}
By the $\RR^k$-flat extension theorem answering the following question affirmatively would also bring an answer to the previous one:
\begin{que}
    Can every $k$-regulous function defined on a variety of dimension greater than 2 be written as a finite sum of $\RR^k$-flat functions?
\end{que}

All the proofs in the paper that require the assumption on the dimension  being equal to two are contained in Section \ref{RR^k-flat decomposition in dimension 2}. One obstacle preventing them from working in higher dimensions is that the situation there becomes more involved as it is harder to control the blowups which may admit centres of arbitrary dimension. Nonetheless it seems that there might be a deeper problem with reproducing the reasoning as shown by the following example:
\begin{ex}
Let $f:=\frac{x^4}{x^2+z^2},g:=\frac{x^4}{x^4+y^2}\in\R(x,y,z)$. Then:
\begin{enumerate}
    \item $f$ is $1$-regulous,
    \item $fg$ is \emph{not} $1$-regulous,
    \item there exists an admissible sequence of two blowups after which $h$ becomes regular and $f$ becomes relatively $\RR^1$-flat.
\end{enumerate}
This shows that Proposition \ref{mult_stays_kregulous} does not hold in the three-dimensional case.
\begin{proof}
The first point is trivial. The second one follows from computing the partial derivative of the product with respect to $y$:
\begin{equation*}
    \pdv{fg}{y}=-\frac{2x^8y}{(x^2+z^2)(x^4+y^2)^2}
\end{equation*}
After substituting $(x,y,z)=(t,t^2,t)$ we get that the limit of the above expression at $0$ is $-\frac{1}{4}$ yet after substituting $(x,y,z)=(0,t,t)$ it is equal to $0$.

It is easy to check that the function $\frac{x^4}{x^4+y^2}\in\R(x,y)$ becomes regular after a sequence of two blowups, the first one at $(0,0)$ and the second one at a point of the exceptional divisor of the first one. This sequence induces a sequence in three dimensions, let the blowups and their exceptional divisors be called $\sigma_1,D_1$ and $\sigma_2,D_2$ respectively. $\sigma_1$ is the blowup at $\Z(x,y)$ while $\sigma_2$ is a blowup at some one dimensional variety contained in $D_1$. We have $\sigma_1(D_1)=\sigma_1\circ\sigma_2(D_2)=\Z(x,y)$. This implies that
\begin{enumerate}
    \item $\ord_{D_1}(x)>0,\ord_{D_2}(x)>0$
    \item $\ord_{D_1}(x^2+z^2)=\ord_{D_2}(x^2+z^2)=0$
\end{enumerate}
Hence $\ord_{D_1}(f)\geq 4,\ord_{D_2}(f)\geq 4$. On the other hand by Proposition \ref{recursive_k_mult} $k_{D_1}\leq 1, k_{D_2}\leq 2$ so $f$ indeed becomes relatively $\RR^1$-flat after the two blowups.
\end{proof}
\end{ex}
This example does not necessarily imply that the entirety of the approach of Section \ref{RR^k-flat decomposition in dimension 2} cannot be made more-than-two-dimensional, although it surely does require much reworking.

Let us also give some comments on our results involving sums of squares. Let $k\geq 2$ and consider the following example:
\begin{ex}
There exists a homogeneous positive definite form of degree $2k$ in $4$ variables which is not a sum of squares of polynomials (it follows for example from the fact that the set of sums of squares is closed in the space of $2k$-forms with standard topology together with a famous result of Hilbert (\cite{Reznick})). By \cite[Theorem 3.2]{sums_of_even_powers} the form is not a sum of squares of $k$-regulous functions. 
\end{ex}
In this example Theorem \ref{sos2} applies proving that the polynomial is a sum of squares of $k-1$-regulous functions giving the optimal class of smoothness. 

By \cite{Becker} the polynomial is a sum of $2k$-th powers of rational functions, which shows that the assumption on the derivative of $f$ in the theorem cannot be dropped whenever $k>1$. It would be interesting to decide whether such a statement holds when $k=1$ as well, in other words:
\begin{que}
    Can every positive semi-definite regular function be written as a sum of squares of $1$-regulous functions?
\end{que}
The author believes that an affirmative answer to this question is plausible as it is only a slight improvement on Theorem \ref{sos1}. 
\section{Acknowledgments}
The author would like to thank his tutor Professor Wojciech Kucharz, his former tutor Doctor Tomasz Kowalczyk and the entire Department of Singularity Theory of Jagiellonian University for introducing him into the topic of regulous functions, and all the feedback regarding the paper they have given him.
\bibliographystyle{plain}
\bibliography{references}
\end{document}